\documentclass[11pt]{article}
\usepackage{graphicx,mathrsfs,amssymb}
\usepackage{amsmath,latexsym, color, amsbsy,amssymb}
\usepackage{graphicx}
\usepackage{xcolor}
\makeatletter
\newcommand{\superimpose}[2]{%
  {\ooalign{$#1\@firstoftwo#2$\cr\hfil$#1\@secondoftwo#2$\hfil\cr}}}
\makeatother

\def\ds{\displaystyle}
\def\forall{\hbox{for all}~}
\def\L{{\bf L}}

\def\ve{\varepsilon}
\def\n{\noindent}

\def\dint{\int\!\!\int}

\def\R{\mathbb{R}}

\def\TV{\hbox{Tot.Var.}}

\def\tv{\hbox{Tot.Var.}}
\def\vs{\vskip 2em}

\def\v{\vskip 1em}

\def\C{{\cal C}}
\def\bega{\begin{array}}
\def\enda{\end{array}}
\def\begi{\begin{itemize}}
\def\endi{\end{itemize}}

\def\Tilde{\widetilde}

\def\bel{\begin{equation}\label}
\def\eeq{\end{equation}}
\def\sqr#1#2{\vbox{\hrule height .#2pt
\hbox{\vrule width .#2pt height #1pt \kern #1pt
\vrule width .#2pt}\hrule height .#2pt }}
\def\square{\sqr74}
\def\endproof{\hphantom{MM}\hfill\llap{$\square$}\goodbreak}

\newtheorem{theorem}{Theorem}[section]

\newtheorem{definition}[theorem]{Definition}
\newtheorem{remark}[theorem]{Remark}
\newtheorem{lemma}[theorem]{Lemma}

\begin{document}
\title{\bf On  scalar nonlinear balance laws with  singular nonlocal sources}\vs
\author{\it Evangelia Ftaka and Khai T. Nguyen\\
		\\
		{\small Department of Mathematics, North Carolina State University}\\
		{\small e-mails:  ~eftaka@ncsu.edu, ~khai@math.ncsu.edu}
	\quad\\
\quad\\
{\footnotesize\it Dedicated to Professor Alberto Bressan on the occasion of his 70 birthday}		
	}

	\maketitle
	\begin{abstract}	
	We investigate one-dimensional scalar balance laws with singular convolution-type source terms. Under appropriate convexity and kernel assumptions, we establish the global existence of entropy weak solutions in ${\bf L}^2(\mathbb{R})$, together with two partial uniqueness results, in the ${\bf L}^2$-periodic setting and non-periodic setting with ${\bf L}^1(\mathbb{R})$ kernel. In the ${\bf L}^1$-kernel case, the characteristic speed satisfies an Oleinik-type estimate, and entropy weak solutions possess locally bounded fractional variation for all positive times. Furthermore, we derive a simple criterion characterizing local smoothness and wave breaking of solutions, which, in particular, includes both the Burger-Poisson and the Burgers-Hilbert equation as  special cases.

		\quad\\
		\quad\\
		{\footnotesize
			\medskip
			
			\n {\bf AMS Mathematics Subject Classification.}  35B65, , 35L60, 35L03, 76B15
		}
	\end{abstract}

\maketitle

\section{Introduction}
\label{sec:1}
\subsection{General setting} Consider a scalar balance law in one space dimension with  a singular source term
\bel{BL}
u_t+f(u)_x~=~{\bf G}[u],
\eeq
where  $u:[0,\infty[\times\R\to\R$ is the state variable and   $f:\R\to\R$ is the $\mathcal{C}^1$ strictly convex flux satisfying the following conditions:
\begin{itemize}
\item [{\bf (F)}] For some constant $C_f>0$ and exponents  $1\leq p_1\leq p_2< p_1+2$,
\[
\begin{cases}
|f'(s)|~\leq~C_f(1+|s|)^{p_2},\qquad \ds\inf_{a\in \R, s>0} \left({f'(a+s)-f'(a)\over s^{p_1}}\right)~\geq~C_f,\\[3mm]
\ds\sup_{|u|,|v|\leq M, u\neq v}{|f(u)-f(v)|\over |f'(u)-f'(v)|}~\doteq~\Lambda_M~<~\infty\qquad\forall M>0.
\end{cases}
\]
\end{itemize}
The singular integral of convolution type ${\bf G}$ is   defined by convolution with a kernel $K$ that is locally integrable on $\R\backslash\{0\}$, in the sense that
\bel{G}
{\bf G}[g](x)~=~\lim_{\ve\to 0+}\int_{|y-x|>\ve}K(x-y)\cdot g(y)~dy.
\eeq
Typically in applications, we shall assume that 
\begin{itemize}
\item [\bf (K1)] ${\bf G}:{\bf L}^{2}(\R)\to {\bf L}^2(\R)$ is a bounded linear operator and 
\bel{C-G}
\|{\bf G}[g]\|_{{\bf L}^2(\R)}~\leq~\|{\bf G}\|_{\infty}\cdot \|g\|_{{\bf L}^2(\R)}\qquad\forall g\in {\bf L}^2(\R);
\eeq
\item [{\bf (K2)}] The kernel $K$ is in $\mathcal{C}^1(\R\backslash \{0\})$ and satisfies
\bel{K1}
\left|K^{(i)}(x)\right|~\leq~{C_K\over |x|^{i+1}}\qquad\text{for} \,\, i=0,1.
\eeq
\end{itemize}
The continuity condition {\bf (K1)} is equivalent to requiring that the Fourier transform of $K$ is bounded in ${\bf L}^{\infty}(\R)$. In particular, {\bf (K1)} holds if the kernel $K$ can be decomposed as
\[
K=K_1+K_2,\qquad K_2~\in~{\bf L}^1(\R),
\]
and the singular part $K_1$ is odd and satisfies {\bf (K2)}. Equation (\ref{BL}) exhibits an interesting structure: while scalar conservation laws generate a contractive semigroup on ${\bf L}^1(\R)$, the operator ${\bf G}$ may be discontinuous and unbounded as an operator on ${\bf L}^1(\R)$. In the archetypal case 
$$
f(u)~=~{u^2\over 2},\qquad\quad K(x)~=~\ds {1\over \pi x},
$$
(\ref{BL}) reduces to the Burgers-Hilbert equation, introduced by Biello and Hunter in \cite{BiH} as a model for surface waves with constant frequency. For this equation, several results are known: lower bounds on the maximal existence time for smooth solutions have been obtained in \cite{BiH, HI, HIDT}; the formation of singularities and the local asymptotic behavior of solutions up to shock formation in finite time were analyzed in \cite{CST, Y}. More recently, piecewise regular solutions with a single shock for “well-prepared’’ initial data were constructed in \cite{BZ, BNK, KV}. For general initial data $\bar u \in {\bf L}^2(\R)$, the global existence of entropy weak solutions to the Burgers-Hilbert equation was proved in \cite{BN}, along with a partial uniqueness result. However, the full well-posedness of the problem remains largely open.
\subsection{Summary of main results}
The present paper first aims to establish the well-posedness and generalized BV regularity properties of entropy weak solutions to the nonlocal balance law (\ref{BL}). We begin by recalling the definition of an entropy weak solution.
\begin{definition}\label{Def1} {\it By an {\bf entropy weak solution} of (\ref{BL})
we mean a function $u\in \L^1_{loc}([0,\infty[\,\times\R)$ with the following properties.
\begi
\item[(i)] The map $t\mapsto u(t,\cdot)$ is continuous with values in $\L^2(\R)$;
\item[(ii)]  For any $k\in \R$ and every nonnegative 
test function $\phi\in\C^1_c(]0,\infty[\,\times\R)$
one has
\bel{ei2}\dint  \Big[ |u-k|\phi_t + \Big(f(u)-f(k)\Big)\hbox{\rm sign}(u-k) \phi_x
+{\bf G}[u(t,\cdot)](x)\hbox{\rm sign}(u-k)\phi  \Big]\, dxdt~\geq~0.\eeq
\endi
}
\end{definition}
{\bf $\bullet$ Global existence of entropy solutions.} Our first main result establishes the global existence of entropy  solutions to the nonlocal balance law (\ref{BL})-(\ref{G}) with ${\bf L}^2(\R)$ initial data.
\begin{theorem}\label{Main1} Assume that   {\bf (F)}  and  {\bf (K1)}-{\bf (K2)} hold. For every  initial data $\bar{u}\in {\bf L}^2(\R)$, the Cauchy problem (\ref{BL})-(\ref{G}) with $u(0,\cdot)=\bar{u}$ admits an entropy weak solution $u(t,x)$ defined for all $(t,x)\in [0,\infty)\times \R $ such that 
\bel{b-u}
\|u(t,\cdot)\|_{{\bf L}^2(\R)}~\leq~e^{\| {\bf G}\|_{\infty}t} \cdot\|\bar{u}\|_{{\bf L}^2(\R)},
\eeq
and  $u(t,\cdot)$ is in ${\bf L}^\infty(\R)$ for each $t>0$.
\end{theorem}
The entropy weak solutions are constructed via a flux-splitting method. Relying on the fractional BV regularity of the semigroup $(S_t)_{t\geq 0}$ generated by the conservation law $u_t + f(u)_x = 0$ with strictly convex flux $f$, we show that the sequence of approximate solutions is precompact and admits a convergent subsequence in ${\bf L}^1_{\mathrm{loc}}(\R)$. To establish compactness in ${\bf L}^2(\R)$, we derive an ${\bf L}^{\infty}$ bound (Lemma \ref{L-infty-b}) and prove the tightness of the approximating solutions (Lemma \ref{TN1}). Achieving these estimates requires a careful analysis of the associated characteristics, which we show to be uniformly H\"older continuous in time.
\medskip

{\bf $\bullet$ Uniqueness and BV regularity.} Uniqueness of weak entropy solutions to (\ref{BL})-(\ref{G}) is more delicate. Indeed, the semigroup $(S_t)_{t\geq 0}$ is contractive with respect to the ${\bf L}^1$-distance but not the ${\bf L}^2$-distance, while the operator ${\bf G}$ is bounded in ${\bf L}^2(\R)$ but not in ${\bf L}^1(\R)$. Consequently, uniqueness cannot be established solely from Lipschitz continuity in ${\bf L}^1(\R)$ or ${\bf L}^2(\R)$; it must instead rely on the specific properties of the singular kernel $K$. In section \ref{unique}, we first address (\ref{BL}) in the periodic setting, 
where the kernel  $K\in\mathcal{C}^1(\R\backslash 2P\mathbb{Z})$ is an odd $2P$-periodic function  and satisfies 
\[
\big|K^{(i)}(x)\big|~\leq~{C_K\over |x|^{i+1}},\qquad x\in [-P,P]\backslash\{0\},~ i\in \{0,1\}.
\]
In Theorem \ref{p-uniqueness},  we establish a uniqueness result for spatially periodic solutions with locally bounded variation. The proof employs Jensen’s inequality together with Lemma \ref{BV-L1}, which provides an estimate for the ${\bf L}^1$-norm of the linear operator ${\bf G}$ applied to a periodic function in terms of its total variation over one period. However, the uniqueness of solutions to (\ref{BL})–(\ref{G}) remains open when the kernel is nonperiodic and lies in ${\bf L}^2(\R)$.

Next, in Theorem \ref{Un-L1}, we prove a well-posedness result for entropy weak solutions to the nonlinear balance law (\ref{BL})-(\ref{G}) with initial data $\bar{u}\in {\bf L}^1(\R)$, assuming the kernel $K\in {\bf L}^1(\R)$ satisfies
\[
\|K\|_{{\bf L}^1(\R)} \leq L_K, \qquad \big|K^{(i)}(x)\big| \leq \frac{C_K}{|x|^{i+1}}, \qquad \text{for} \,\, i=0,1.
\]
Moreover, the characteristic speed satisfies an Oleinik-type estimate in (\ref{dcd2}), and the unique entropy weak solution $u(t,\cdot)$ belongs to a locally bounded fractional total variation space for every time $t>0$.

{\bf $\bullet$ Local smoothness and wave breaking.} In Section \ref{S4}, we study the local smoothness and wave breaking for the nonlocal balance law (\ref{BL})-(\ref{G}). This problem was first posed by Whitham \cite{W} and has been investigated for nonlocal perturbations of the Burgers equation with various kernels $K$ (see, e.g., \cite{CE, Hu1, Hu, SW1,RY, Liu, G}). Of particular interest is the case of a singular kernel, $K(0)=\infty$, where the analysis is substantially more delicate. A prototypical example is the Burgers–Hilbert equation
\bel{BH}
u_t+u u_x~=~{1\over \pi}\cdot \lim_{\ve\to 0+}\int_{|y-x|>\ve}{u(t,y)\over {x-y}}~dy.
\eeq
for which rigorous proofs of wave breaking have recently been obtained \cite{CST, SW, Y}, as observed in the numerical simulations of \cite{BiH}. Building on these results, we extend the analysis to a broader class of singular convolution odd kernels, deriving precise and practical criteria for local smoothness and the onset of wave breaking for weak entropy solutions of (\ref{BL}). In Theorem \ref{main4}, we establish a sharp criterion for strictly convex fluxes 
$f \in \mathcal{C}^3(\mathbb{R})$ satisfying, for some constants $\Gamma \ge 0$ and 
$p \in [0, 2/3)$, that
\[
|f'''(u)|~ \le~ \Gamma \big(1 + |u|^p\big), \qquad \forall\, u \in \mathbb{R}.
\]
We note that the upper bound $2/3$ on the  exponent $p$ in the growth condition on $f'''$ used in Theorem~\ref{main4} is sharp for our analysis, ensuring that the wave-breaking conditions do not depend on the bound $\|u\|_{{\bf L}^{\infty}(\mathbb{R})}$. In Theorem \ref{general-f}, we extend the result to general strictly convex fluxes.   

In the  case of  quadratic fluxes 
\bel{f-q}
f(u)~ = ~a u^2 + b u + c, \qquad a > 0,
\eeq
we obtain the following theorem, which gives explicit wave-breaking criteria in this simplified setting and applies to both the Burger–Poisson and Burgers–Hilbert equation (\ref{BH}).  
\begin{theorem}\label{BHG1}
Assume that $f$ takes the quadratic form in (\ref{f-q}) and kernel  $K$ is odd and satisfies {\bf(K2)}. Then, for every $\theta \in (0,1/4]$ and every $\bar{u} \in H^2(\R)$ satisfying
\bel{B-Con-0}
\Big|\inf_{x \in \R} \bar{u}'(x)\Big| ~>~ \frac{2^{3/4}a^{1/2}}{\theta^{1/2}} \, \|\bar{u}'\|^{1/4}_{{\bf L}^2(\R)} \, \|\bar{u}''\|^{1/4}_{{\bf L}^2(\R)},
\eeq
the Cauchy problem (\ref{BL})-(\ref{G}) with initial data $u(0,\cdot) = \bar{u}(x)$ exhibits wave breaking at some finite time $T^* > 0$ such that
\bel{b-T*-0}
\frac{1}{2a(1+\theta)} \cdot \frac{1}{\Big|\ds \inf_{x \in \R} \bar{u}'(x)\Big|} ~\le~ T^*~\le~ \frac{1}{2a(1-\theta)} \cdot \frac{1}{\Big|\ds\inf_{x \in \R} \bar{u}'(x)\Big|} \,.
\eeq
\end{theorem}
As a consequence of  Theorem \ref{main4}, in Remark \ref{Lset}, we show that for every $T>0$, there exists a quite large open set $\mathcal{B}_T$ of $H^2(\R)$ such that  the Cauchy problem \textup{(\ref{BL})}-(\ref{G})  with initial data $\bar{u} \in \mathcal{B}_T$ exhibits wave breaking before time $T$.
The proof of Theorem \ref{BHG1}, Theorem \ref{main4}, Theorem~\ref{general-f} is subtle yet straightforward and is based on analyzing the blow-up of the quantity
\[
m(t)~\doteq~ \Big|\inf_{x \in \R} [f'(u(t,x))]_x\Big|.
\]
By tracking $[f'(u)]_x$ along the characteristic curves of (\ref{BL})-(\ref{G}) and differentiating the equation with respect to $x$, we obtain a system of ODEs for the ${\bf L}^2$-norms of $u_x$, $u_{xx}$, and $[f'(u)]_x$. 
This system captures the growth of the gradient and allows us to determine precisely when wave breaking occurs, linking it to the initial slope and the influence of the nonlocal term.

\section{Global existence of entropy weak solutions}
\label{sec:2}
\setcounter{equation}{0}
The existence result of entropy weak solutions to the nonlocal balance law (\ref{BL})-(\ref{G}) with initial data in ${\bf L}^{2}(\R)$ is based on a limiting process
for approximating solutions, which are constructed by the flux-splitting method. Toward to the construction, recalling a well-known result (see e.g. in \cite{Bbook}), let  $(S_t)_{t\geq 0}$ be a nonlinear contractive semigroup such that $t\mapsto S_t\bar{v}$ is the unique entropy weak solution to 
\bel{CL}
v_t+f(v)_x~=~0,\qquad\qquad  v(0,\cdot)~=~\bar{v}~\in~{\bf L}^{1}(\R)\cap {\bf L}^{\infty}(\R). 
\eeq
Since $\eta(v)=\ds{v^2\over 2}$ is a convex entropy for (\ref{CL}) and  $q:\R\to\R$ defined by 
\bel{q}
q(0)~=~0,\qquad q(v)~=~\int_{0}^{v}sf'(s)ds\qquad\forall v\in \R\backslash\{0\}
\eeq
is the associated entropy flux, every admissible solution satisfies the following inequality in a distributional sense
\[
\left({v^2\over 2}\right)_t+q(v)_x~\leq~0.
\]
 In particular, for every $\bar{v}\in {\bf L}^{2}(\R)$, the map $t\mapsto \|S_t\bar{v}\|_{{\bf L}^2(\R)}$ is nondecreasing in $[0,+\infty[$, and 
\bel{L2-b}
 \|S_t\bar{v}\|_{{\bf L}^2(\R)}~\leq~\|\bar{v}\|_{{\bf L}^2(\R)}\qquad\forall t\geq 0.
\eeq
Since $f$ is strictly convex, the weak entropy solution $v(t,\cdot)\doteq S_t\bar{v}$ of (\ref{CL}) satisfies a type of Oleinik estimate 
\bel{O}
f'\big(v(t,y)\big)-f'\big(v(t,x)\big)~\leq~{y-x\over t}\qquad\forall y> x,t>0.
\eeq
Thus, $f'\big(v(t,\cdot)\big)$ is in $BV_{\mathrm{loc}}(\R)$ and   at any discontinuous point $x\in \R$ the right and left limits $v(t,x\pm)=\ds \lim_{y\to x\pm}v(t,y)$ exist and 
\[
 \lim_{y\to x+}v(t,y)~=~v(t,x+)~<~v(t,x-)~=~\lim_{y\to x-}v(t,y).
\]
Relying on the  estimates (\ref{L2-b}) and (\ref{O}),  we  shall derive  a fractional BV bound on $v(t,\cdot)$ for all $t>0$.
\begin{definition} 
Let $g:[a,b]\to\R$ be a real-valued function. For every $\gamma\in ]0,1]$, the fractional $\gamma$-total variation of $g$ on $[a,b]$ is defined by 
\[
TV^{\gamma}\{g;[a,b]\}~\doteq~\sup_{a=x_0<x_1< \cdots < x_N=b, N\geq 1} \left(\sum_{i=1}^{N}|g(x_{i})-g(x_{i-1})|^{1/\gamma}\right).
\]
We say that $g$ belongs to $BV^{\gamma}([a,b])$ if and only if $TV^{\gamma}\{g;[a,b]\}<\infty$.
\end{definition}

%
%
\begin{lemma}\label{F-BV} For every $t>0$, $v(t,\cdot)$ is in ${\bf L}^{\infty}(\R)$ with 
\bel{L-b-v}
\|v(t,\cdot)\|_{{\bf L}^{\infty}(\R)}~\leq~M_t~\doteq~4\cdot\left({\|\bar{v}\|^2_{{\bf L}^2(\R)}\over C_ft}\right)^{1/(p_1+2)}\qquad\forall t>0.
\eeq
Moreover,  $v(t,\cdot)$ is in $BV^{1/p_1}_{\mathrm{loc}}(\R)$ and satisfies
\bel{BV-b}
TV^{1/p_1}(v(t,\cdot);[a,b])~\leq~{2\over C_f}\cdot \left({b-a\over t}+\|f'\|_{{\bf L}^{\infty}([-M_t,M_t])}\right)
\eeq
 for every $-\infty<a<b<+\infty$.
\end{lemma}
{\bf Proof.} {\bf 1.} Let $\Phi_f:[0,\infty[\to [0,\infty[$ be a strictly increasing function which is defined by  
\bel{Phi}
\Phi_f(s)~\doteq~\inf_{a \in\R} \big\{f'(a+s)-f'(a)\big\}~\geq~C_f s^{p_1}\qquad\forall s>0.
\eeq
Assume that $\ds {m\over 2}\leq v(t,y_0)\leq m$ for some $y_0\in\R$ and $m>0$. By a type of Oleinik estimate in (\ref{O}) and (\ref{Phi}), we derive
\[
v(t,y_0)-v(t,x)~\leq~\Phi_f^{-1}\left({y_0-x\over t}\right),\qquad x\in (-\infty,y_0].
\]
In particular, for every $y_0-t\cdot\Phi_f({m\over 4})<x<y_0$, it holds 
\[
v(t,x)~\geq~v(t,y_0)-\Phi_f^{-1}\left({y_0-x\over t}\right)~\geq~{m\over 2}-\Phi_f^{-1}\left(\Phi_f\left({m\over 4}\right)\right)~=~{m\over 4}.
\]
Then, recalling (\ref{L2-b}), we get
\[
\|\bar{v}(\cdot)\|^2_{{\bf L}^2}~\geq~\|v(t,\cdot)\|^2_{{\bf L}^2}~\geq~\int^{y_0}_{y_0-t\cdot\Phi_f(m/4)}|v(t,x)|^2dx~\geq~{t{m^2}\over 16}\Phi_f\left({m\over 4}\right)~\geq~{C_ft\over 4^{2+p_1}}\cdot m^{2+p_1},
\]
which yields 
\[
v(t,y)~\leq~M_t\qquad\forall y\in \R.
\]
Similarly, one can show that $v(t,y)\geq-M_t$ and  obtain (\ref{L-b-v}).
\medskip

{\bf 2.} For every given $a<b$, let $\{x_0,x_1,\dots, x_N\}$ be a partition of [a,b]. Without loss of generality, assume that $v$ is continuous at all points $x_i$. Set $\beta_i\doteq v(\tau,x_i) - v(\tau,x_{i-1})$. For every $i\in\{1,\cdots, N\}$, we  define
\[
I^+~=~\{i\in \{1,\dots, N\}: \beta_i\geq 0\},\qquad I^-~=~\{i\in \{1,\dots N\}: \beta_i<0 \}.
\]
From (\ref{O}) and (\ref{Phi}), one gets
\[
\Phi_f(\beta_i)~\leq~f'\big(v(\tau,x_i)\big)-f'\big(v(\tau,x_{i-1})\big)~\leq~{x_i-x_{i-1}\over \tau},\qquad i\in I^+,
\]
and this implies that 
\[
\sum_{i\in I^+}\Phi_f(\beta_i)~\leq~\sum_{i\in I^+}{x_i-x_{i-1}\over \tau}~\leq~{b-a\over \tau}.
\]
On the other hand, we denote by  $\xi_i(\cdot)$ the backward
characteristic line $t\mapsto \xi_i (t)$ through $(\tau,x_i)$ such that 
\[
x_i ~=~\xi_i(t)+f'\big(v(\tau,x_i)\big)\cdot (\tau-t),\qquad u(t,\xi_i(t))~=~v(\tau,x_i),\qquad t\in [0,\tau].
\]
For every $i\in I^{-}$, we have 
\[
0~\leq~x_{i}-x_{i-1}~=~\xi_i(0)-\xi_{i-1}(0)+\big[f'\big(v(\tau,x_i)\big)-f'\big(v(\tau,x_{i-1})\big)\big]\cdot \tau,
\]
and this yields 
\[
\begin{split}
\sum_{i\in I^{-}}\Phi_f(|\beta_i|)&~\leq~\sum_{i\in I^{-}}{\xi_i(0)-\xi_{i-1}(0)\over \tau}~\leq~{\xi_N(0)-\xi_1(0)\over \tau}~\leq~{b-a\over \tau}+2 \|f'\|_{{\bf L}^{\infty}([-M_t,M_t])}.
\end{split}
\]
By (\ref{Phi}), we get 
\[
C_f\cdot \sum_{i=1}^{N}|\beta_i|^{p_1}~\leq~\sum_{i\in \{1,2,\dots,N\}}\Phi_f(\beta_i)~\leq~{2(b-a)\over \tau}+2 \|f'\|_{{\bf L}^{\infty}([-M_t,M_t])},
\]
and this yields (\ref{BV-b}).
\endproof

\begin{remark}\label{L1-1} Assume that $\bar{v}\in {\bf L}^1(\R)$. Then, by the same argument as in the proof of Lemma \ref{F-BV}, we obtain 
\[
\|v(t,\cdot )\|_{{\bf L}^{\infty}(\R)}~\leq~4\cdot\left({\|\bar{v}\|_{{\bf L}^1(\R)}\over C_ft}\right)^{1/(p_1+1)}
\]
for all $t>0$.
\end{remark}

\subsection{Approximate solutions by a flux-splitting method}
Fix an initial datum $\bar{u}\in {\bf L}^2(\R)$. We shall construct a solution of (\ref{BL})-(\ref{G}) for $t\in [0,1]$. By repeating the procedure, the solution can then be prolonged to any time interval $[0,T]$ for $T>0$.
 For every given  integer $\nu\geq 1$, consider the time steps
\[
t_\ell~\doteq~\ell\cdot 2^{-\nu},\qquad\qquad \ell=0,1,2,\dots
\]
The approximate solution of (\ref{BL}) is inductively defined as
\bel{u-nu}
\begin{cases}
u^{\nu}(0,\cdot)~=~\bar{u},\qquad u^{\nu}(t_{\ell},\cdot)~=~u^{\nu}(t_\ell-,\cdot)+2^{-\nu}\cdot {\bf G}[u^{\nu}(t_\ell-,\cdot)],\qquad \ell=1,2,\dots,\\[4mm]
u^{\nu}(t,\cdot)~=~S_{t-t_\ell}\big(u^{\nu}(t_\ell,\cdot)\big),\qquad t\in [t_\ell,t_{\ell+1}[,\quad \ell=0,1,2,\dots
\end{cases}
\eeq
From (\ref{L2-b}), one has  
\[
\|u^{\nu}(t_\ell-,\cdot)\|_{{\bf L}^2(\R)}~=~\left\|S_{2^{-\nu}}\big(u^{\nu}(t_{\ell-1},\cdot)\big)\right\|_{{\bf L}^2(\R)}~\leq~\|u^{\nu}(t_{\ell-1},\cdot)\|_{{\bf L}^2(\R)},
\]
and derives from (\ref{C-G}) that  
\[
\begin{split}
\big\|u^{\nu}(t_\ell,\cdot)\big\|_{{\bf L}^2} & ~\leq~ \bigl(1+2^{-\nu} \, \|{\bf G}\|_{\infty}\bigr) \cdot \|u^{\nu}(t_\ell-,\cdot)\|_{{\bf L}^2(\R)} ~\leq~ \bigl(1+2^{-\nu} \, \|{\bf G}\|_{\infty}\bigr) \cdot \|u^{\nu}(t_{\ell-1},\cdot)\|_{{\bf L}^2(\R)}\\[2mm]
& ~\leq~ \bigl(1+2^{-\nu} \, \|{\bf G}\|_{\infty}\bigr)^{2^{\nu}t_\ell} \cdot \|\bar{u}\|_{{\bf L}^2(\R)}.
\end{split}
\]
Thus, (\ref{L2-b}) yields
\bel{L2-u}
\|u^{\nu}(t,\cdot)\|_{{\bf L}^2(\R)}~\leq~ \bigl(1+2^{-\nu} \, \|{\bf G}\|_{\infty}\bigr)^{2^{\nu}t} \cdot \|\bar{u}\|_{{\bf L}^2(\R)} ~\leq~ e^{\| {\bf G}\|_{\infty}t} \cdot \|\bar{u}\|_{{\bf L}^2(\R)} \qquad\forall t\geq 0.
\eeq
Taking $\nu\to\infty$, we then obtain that 
\bel{e-s-l2}
\limsup_{\nu\to\infty}\|u^{\nu}(t,\cdot)\|_{{\bf L}^2(\R)}~\leq~e^{\| {\bf G}\|_{\infty}t} \cdot\|\bar{u}\|_{{\bf L}^2(\R)}.
\eeq
In the next steps we will show that the sequence of flux-splitting approximations $(u^{\nu})_{\nu\geq 1}$ is precompact and has a convergent subsequence in ${\bf L}^1_{loc}(\R)$. This will follow from the decay properties of entropy weak solutions to (\ref{CL}) with strictly convex flux $f$. In order to do so, we first establish the H\"older continuity on backward generalized characteristics. 
\subsubsection{Generalized characteristics}
For a fixed $\nu>0$ and $\bar{u}\in{\bf L}^{2}(\R)$, we define the generalized characteristics associated to the approximate solution $u^{\nu}$ which is previously constructed. 
\begin{definition}\label{Char} An absolutely continuous map $t\mapsto x(t)$ in $[a,b]$ is called a generalized characteristic associated to the approximate solution $u^{\nu}$ if it satisfies the
differential inclusion
\bel{C-Diff-In}
\dot{x}(t)~\in ~\big[f'\big(u^{\nu}(t,x(t)+)\big), f'\big(u^{\nu}(t,x(t)-)\big)\big]\qquad \text{for a.e.} ~t\in [a,b].
\eeq
A characteristic is called genuine if $u^{\nu}(t,x(t)+)=u^{\nu}(t,x(t)-)$ for almost all $t\in [a,b]$.
\end{definition}
Notice that from  Lemma \ref{F-BV}, the function  $u^{\nu}(t,\cdot)$ is  in $BV^{1/p_1}_{\mathrm{loc}}(\R)$ at any  time  $t\neq t_\ell$. In this case,   the right and left limits in (\ref{C-Diff-In}) are well defined. Moreover, as proved by Dafermos \cite{D}, the minimal and maximal backward characteristics through any given point are always genuine. Next,  we set 
\bel{gp}
\gamma_p~\doteq~{p_2\over 2+p_1}~\in ~ (0,1),
\eeq
and  establish a H\"older continuity on  the backward characteristics.

\begin{lemma}\label{hc-bb-u} 
For every given $\nu>0$,  let $t\mapsto x(t)$ be any characteristic on $[0,1]$ for the approximate solution $u^{\nu}$. For every $0\leq \tau_1<\tau_2\leq 1$, it holds 
\bel{Holder}
|x(\tau_2)-x(\tau_1)|~\leq~\Gamma_1\cdot (\tau_2-\tau_1)^{1-\gamma_p},
\eeq
where the constant $\Gamma_1\geq 0$ is computed by 
\bel{gamma-T}
\Gamma_1~\doteq~C_f\cdot \left[2^{1+p_1}+{2^{4+2p_1}e^{2\|{\bf G}\|_{\infty}}\over C_f}\cdot \big(1 +\|{\bf G}\|_{\infty}+ \|{\bf G}\|^2_{\infty}\big)  \cdot \|\bar{u}\|^2_{{\bf L}^2(\R)}\right]^{\gamma_p}.
\eeq
\end{lemma}
{\bf Proof.} Consider the positive and negative part of the initial data
\[
\bar{u}(x)~=~\bar{u}^+(x)+\bar{u}^{-}(x)\qquad\mathrm{with}\qquad \begin{cases}
\bar{u}^{+}(x)~=~\max\{\bar{u}(x),0\},\\[3mm]
\bar{u}^{-}(x)~=~\min\{\bar{u}(x),0\}.
\end{cases}
\]
Splitting  the source term into its positive and negative parts  
\[
G^{\nu}(t_\ell,x)~\doteq~ {\bf G}[u^{\nu}(t_\ell-,\cdot)](x)~=~G^{\nu,+}(t_\ell,x)+G^{\nu,-}(t_\ell,x),
\]
we  define the functions $u^{\nu,+},u^{\nu,-}$ inductively by setting
\[
\begin{cases}
u^{\nu,\pm}(0,\cdot)~=~\bar{u}^{\pm}(\cdot),\qquad u^{\nu,\pm}(t_\ell,\cdot)~=~u^{\nu,\pm}(t_\ell-,\cdot)+2^{-\nu}\cdot G^{\nu,\pm}(t_\ell,\cdot), \\[3mm]
u^{\nu,\pm}(t,\cdot)~\doteq~S_{t-t_\ell}(u^{\nu,\pm}(t_\ell,\cdot)),\qquad t\in [t_\ell,t_{\ell+1}[.
\end{cases}
\]
By a standard comparison theorem for scalar conservation laws, we get  
\[
u^{\nu,-}(t,\cdot)~\leq~\min\{0,u^{\nu}(t,\cdot)\}~\leq~\max\{0,u^{\nu}(t,\cdot)\}~\leq~u^{\nu,+}(t,\cdot).
\]
Moreover, since $\big\|G^{\nu,\pm}(t_\ell,\cdot)\big\|_{{\bf L}^2(\R)}\leq \|{\bf G}[u^{\nu}(t_\ell-,\cdot)]\|_{{\bf L}^2(\R)}\leq \|{\bf G}\|_{\infty}\cdot e^{\|{\bf G}\|_{\infty}t_\ell}\|\bar{u}\|_{{\bf L}^2(\R)}$, it holds
\[
\big\|u^{\nu,\pm}(t,\cdot)\big\|_{{\bf L}^2(\R)}~\leq~e^{\| {\bf G}\|_{\infty}t} \cdot\|\bar{u}\|_{{\bf L}^2(\R)},\qquad t\geq 0.
\]
Calling $t\mapsto x^{\pm}(t)$  the minimal backward characteristic for $u^{\nu,\pm}$, through the point $\big(\tau_2,x(\tau_2)\big)$, we have that 
\[
x^{+}(t)~\leq~x(t)~\leq~x^{-}(t),\qquad t\in [0,\tau_2].
\]
This particularly yields 
\bel{hol}
x^{-}(\tau_2)-x^{-}(\tau_1)~\leq~x(\tau_2)-x(\tau_1)~\leq~x^{+}(\tau_2)-x^{+}(\tau_1)
\eeq
Two cases are considered: 
\medskip

{\bf $\bullet$ Case 1.} If $x^{-}(\tau_2)-x^{-}(\tau_1)\geq 0$ then 
\[
|x(\tau_2)-x(\tau_1)|~\leq~x^{+}(\tau_2)-x^{+}(\tau_1).
\]
To bound the difference  $x^{+}(\tau_2)-x^{+}(\tau_1)$, we shall estimate the integral of $\left[u^{\nu,+}\right]^2$ over the domain $\{(s,z):s\in [\tau_1,\tau_2],z\geq x^+(s)\}$ by using the divergence theorem for $(\eta,q)$ as follows: 
\[
\begin{split}
0&~\leq~{1\over 2}\int^{\infty}_{x^+(\tau_2)}\left[u^{\nu,+}(\tau_2,z)\right]^2dz~=~\int^{\infty}_{x^+(\tau_2)}\eta\big(u^{\nu,+}(\tau_2,z)\big)dz\\
&~\leq~\int^{\infty}_{x^+(\tau_1)}\eta\big(u^{\nu,+}(\tau_2,z)\big)dz+\sum_{\tau_1<t_\ell\leq \tau_2}\int^{\infty}_{x^+(t_{\ell})}\bigg(\eta\big(u^{\nu,+}(t_\ell,z)\big)-\eta\big(u^{\nu,+}(t_\ell-,z)\big)\bigg)dz\\
&\qquad\qquad\qquad +\int_{\tau_1}^{\tau_2}\bigg(q\Big(u^{\nu,+}\big(s,x^+(s)\big)\Big)-\eta\Big(u^{\nu,+}\big(s,x^{+}(s)\big)\Big)\cdot \dot{x}^+(s)\bigg)ds\\
&~\leq~{e^{2\|{\bf G}\|_{\infty}\tau_2}\over 2}\cdot \|\bar{u}\|_{{\bf L}^2(\R)}^2+{1\over 2}\cdot \sum_{\tau_1<t_\ell\leq \tau_2} e^{2\|{\bf G}\|_{\infty}t_\ell}\left(2^{1-\nu}\cdot \|{\bf G}\|_{\infty}+2^{-2\nu}\|{\bf G}\|^2_{\infty}\right)\cdot \|\bar{u}\|^2_{{\bf L}^2(\R)} \\
&\qquad\qquad\qquad\qquad\qquad\qquad +\int_{\tau_1}^{\tau_2}\bigg(q\Big(u^{\nu,+}\big(s,x^+(s)\big)\Big)-\eta\Big(u^{\nu,+}\big(s,x^{+}(s)\big)\Big)\cdot \dot{x}^+(s)\bigg)ds\\
&~\leq~ \big(1 +\|{\bf G}\|_{\infty}+ \|{\bf G}\|^2_{\infty}\big) e^{2\|{\bf G}\|_{\infty}}  \|\bar{u}\|^2_{{\bf L}^2(\R)} + \int_{\tau_1}^{\tau_2}\bigg(q\Big(u^{\nu,+}\big(s,x^+(s)\big)\Big)-\eta\Big(u^{\nu,+}\big(s,x^+(s)\big)\Big)\cdot \dot{x}^+(s)\bigg)ds,
\end{split}
\]
which implies that 
\[
\begin{split}
\int_{\tau_1}^{\tau_2}\bigg[\eta\Big(u^{\nu,+}\big(s,x^{+}(s)\big)\Big)\cdot f'\Big(u^{\nu,+}\big(s,x^{+}(s)\big)\Big)&-q\Big(u^{\nu,+}\big(s,x^+(s)\big)\Big)\bigg]ds\\
&\leq \big( 1 + \|{\bf G}\|_{\infty}+ \|{\bf G}\|^2_{\infty}\big) e^{2\|{\bf G}\|_{\infty}} \cdot \|\bar{u}\|^2_{{\bf L}^2(\R)}.
\end{split}
\]
Recalling (\ref{q}) and {\bf (F)}, for every $\omega\geq 0$ we have  
\[
\begin{split}
\eta(\omega)f'(\omega)-q(\omega)&~=~\ds{\omega^2f'(\omega)\over 2} - \int_{0}^{\omega}sf'(s)ds~=~{1\over 2}\cdot\int_{0}^{\omega}{s^2}f''(s)ds\\
&~\geq~{\omega^{2}\over 8}\int_{\omega/2}^\omega f''(s)ds~=~{\omega^2\over 2^{3}} \big[f'(\omega)-f'(\omega/2)\big]~\geq~{C_f\over 2^{3+p_1}}\cdot \omega^{2+p_1}\\
&~\geq~{C_f\over 2^{3+p_1}}\cdot \left[{(1+\omega)^{2+p_1}\over 2^{1+p_1}}-1\right]~\geq~{C_f\over 2^{3+p_1}}\cdot\left({|f'(\omega)|^{1/\gamma_p}\over C_f^{1/\gamma_p}2^{1+p_1}}-1\right),
\end{split}
\]
and derive from the above estimate that 
\bel{uet1}
\begin{split}
\int_{\tau_1}^{\tau_2}\bigg({\big|\dot{x}^+(s)\big|^{1/\gamma_p}\over C_f^{1/\gamma_p}2^{1+p_1}}-1\bigg)~ds
&~\leq~ {2^{3+p_1}e^{2\|{\bf G}\|_{\infty}}\over C_f}\cdot \big(1 + \|{\bf G}\|_{\infty}+ \|{\bf G}\|^2_{\infty}\big)  \cdot \|\bar{u}\|^2_{{\bf L}^2(\R)}.
\end{split}
\eeq
By Jensen's inequality, we estimate 
\[
{1\over \tau_2-\tau_1}\int_{\tau_1}^{\tau_2}\big|\dot{x}^+(s)\big|^{1/\gamma_p}ds~\geq~\left({1\over \tau_2-\tau_2}\int_{\tau_1}^{\tau_2}|\dot{x}^+(s)|ds\right)^{1/\gamma_p}~\geq~\left({x^+(\tau_2)-x^+(\tau_1)\over \tau_2-\tau_1}\right)^{1/\gamma_p}.
\]
From (\ref{uet1}), we have 
\[
\begin{split}
{1\over 2^{1+p_1}}\left({x^+(\tau_2)-x^+(\tau_1)\over C_f(\tau_2-\tau_1)}\right)^{1/\gamma_p}&~\leq~{2^{3+p_1}e^{2\|{\bf G}\|_{\infty}}\over C_f}\cdot \big (1 + \|{\bf G}\|_{\infty}+ \|{\bf G}\|^2_{\infty}\big)  \cdot \|\bar{u}\|^2_{{\bf L}^2(\R)}\cdot {1\over \tau_2-\tau_1}+1\\[2mm]
&~\leq~\left(1+{2^{3+p_1}e^{2\|{\bf G}\|_{\infty}}\over C_f}\cdot \big( 1 + \|{\bf G}\|_{\infty}+ \|{\bf G}\|^2_{\infty}\big)  \cdot \|\bar{u}\|^2_{{\bf L}^2(\R)}\right)\cdot {1\over \tau_2-\tau_1},
\end{split}
\]
and this yields 
\[
x^+(\tau_2)-x^+(\tau_1)~\leq~\Gamma_1\cdot (\tau_2-\tau_1)^{1-\gamma_p}.
\]
Thus, (\ref{Holder}) holds.
\medskip 

{\bf $\bullet$ Case 2:} If $x^{+}(\tau_2)-x^{+}(\tau_1)\leq 0$ then 
\[
|x(\tau_2)-x(\tau_1)|~\leq~x^{-}(\tau_1)-x^{-}(\tau_2).
\]
With the same argument as in {\bf Case 1}, we get 
\begin{multline*}
\int_{\tau_1}^{\tau_2}\bigg[-\eta\Big(u^{\nu,-}\big(s,x^{-}(s)\big)\Big)\cdot f'\Big(u^{\nu,-}\big(s,x^{-}(s)\big)\Big)+q\Big(u^{\nu,-}\big(s,x^-(s)\big)\Big)\bigg]ds\\
~\leq~\big(1+\|{\bf G}\|_{\infty}+ \|{\bf G}\|^2_{\infty}\big)\cdot e^{2\|{\bf G}\|_{\infty}}\cdot \|\bar{u}\|^2_{{\bf L}^2(\R)}.
\end{multline*}
Recalling (\ref{q}) and  the assumption {\bf (F)} on $f$, we also derive 
\[
\int_{\tau_1}^{\tau_2}\bigg({|\dot{x}^-(s)\big|^{1/\gamma_p}\over C_f^{1/\gamma_p}2^{1+p_1}}-1\bigg)~ds~\leq~ {2^{3+p_1}e^{2\|{\bf G}\|_{\infty}}\over C_f} \cdot \big(1 + \|{\bf G}\|_{\infty}+ \|{\bf G}\|^2_{\infty}\big)  \cdot \|\bar{u}\|^2_{{\bf L}^2(\R)}.
\]
Again, by Jensen's inequality, we estimate 
\[
{1\over \tau_2-\tau_1}\int_{\tau_1}^{\tau_2}\big|\dot{x}^-(s)\big|^{1/\gamma_p}ds~\geq~\left({1\over \tau_2-\tau_2}\int_{\tau_1}^{\tau_2}|\dot{x}^-(s)|ds\right)^{1/\gamma_p}~\geq~\left({x^-(\tau_1)-x^-(\tau_2)\over \tau_2-\tau_1}\right)^{1/\gamma_p},
\]
and (\ref{Holder}) also  holds in this case. 
\endproof

\subsubsection{${\bf L}^{\infty}$-bound} Toward to  an ${\bf L}^{\infty}$-bound for our constructed  entropy weak solution to (\ref{BL}), for a given $\nu\in\mathbb{Z}^+$, we shall provide an a priori ${\bf L}^{\infty}$-bound for  the sequence of approximate solutions $u^{\nu}$ at times $\ell 2^{-\nu}$ for all 
$\ell\in \{1,\cdots, 2^{\nu}\}$.
\begin{lemma}\label{L-infty-b} There exists a constant $\Gamma_2>0$ such that for every $\nu\in \mathbb{Z}^+$ and $T=\ell_02^{-\nu}$ for $\ell_0\in \{1,\cdots, 2^{\nu}\}$, it holds 
\bel{L-infty}
\left\|u^{\nu}(T-,\cdot)\right\|_{{\bf L}^{\infty}(\R)}~\leq~ \Gamma_2\cdot \left(1+{1\over T^{1/(p_1+2)}}\right).
\eeq
\end{lemma}
{\bf Proof.} {\bf 1.} Fix $\nu>0$  and $T=\ell_02^{-\nu}$ for some $\ell_0\in\{1,\cdots,2^{\nu}\}$. We shall establish an upper bound on $u^{\nu,\pm}(T-,\cdot)$ which is defined in the proof of Lemma \ref{hc-bb-u} such that 
\[
u^{\nu,-}(T-,\cdot)~\leq~u^{\nu}(T-,\cdot)~\leq~u^{\nu,+}(T-,\cdot).
\]
Let  $\big\{u_n :[0,T]\times\R\to\R\big\}_{n\geq 1}$ be a non-decreasing  sequence of functions which are defined by 
\bel{T-un}
u_n(t,\cdot)~=~\begin{cases}
u^{\nu,+}(t,\cdot),\quad t\in [0,\tau_n[,\\[2mm]
S_{t-\tau_n}u^{\nu,+}(\tau_n,\cdot),\quad t\in [\tau_n,T],
\end{cases}
\quad \tau_n~=~T(1-2^{-n}).
\eeq
By Lemma \ref{F-BV} and (\ref{L2-u}), we have 
\bel{et1}
\begin{split}
\|u_1(T-,\cdot)\|_{{\bf L}^{\infty}(\R)}&~\leq~\left\|S_{T/2}u^{\nu,+}(T/2,\cdot)\right\|_{{\bf L}^{\infty}(\R)}~\leq~4\cdot\left({\big\|u^{\nu,+}(T/2,\cdot)\big\|^2_{{\bf L}^2(\R)}\over C_f T/2}\right)^{1/(p_1+2)}\\[3mm]
&~\leq~4\left({e^{\| {\bf G}\|_{\infty}T}  \|\bar{u}\|^{2}_{{\bf L}^2(\R)}
\over C_f T/2}\right)^{1/(p_1+2)}.
\end{split}
\eeq
Let $n_0\in \mathbb{N}$ be such that 
\bel{n0}
T2^{-n_0}~<~2^{-\nu}~\leq ~T2^{-n_0+1}.
\eeq
From the construction (\ref{T-un}), it holds
\[
u_{n_0}(T-,\cdot)~=~u^{\nu,+}(T-,\cdot).
\]
This particularly yields
\bel{sum1}
\begin{split}
\|u^{\nu,+}&(T-,\cdot)\|_{{\bf L}^{\infty}(\R)}~\leq~\|u_{1}(T-,\cdot)\|_{{\bf L}^\infty(\R)}+\sum_{n=2}^{n_0}\left(\|u_{n}(T-,\cdot)\|_{{\bf L}^\infty(\R)}-\|u_{n-1}(T-,\cdot)\|_{{\bf L}^\infty(\R)}\right)\\
&~\leq~4\left({e^{ \|{\bf G}\|_{\infty}T}  \|\bar{u}\|^{2}_{{\bf L}^2(\R)}
\over C_f T/2}\right)^{\frac{1}{p_1+2}}+\sum_{n=2}^{n_0}\big(\|u_{n}(T-,\cdot)\|_{{\bf L}^\infty(\R)}-\|u_{n-1}(T-,\cdot)\|_{{\bf L}^\infty(\R)}\big).
\end{split}
\eeq

{\bf 2.} Next, for every  $2\leq n\leq n_0$, we  provide an upper bound on  $\|u_{n}(T-,\cdot)\|_{{\bf L}^\infty(\R)}-\|u_{n-1}(T-,\cdot)\|_{{\bf L}^\infty(\R)}$. Assume that  
\bel{te-a}
\theta_n~\doteq~\|u_n(T-,\cdot)\|_{{\bf L}^{\infty}(\R)}-\|u_{n-1}(T-,\cdot)\|_{{\bf L}^{\infty}(\R)}-2^{-n}~>~0.
\eeq
Then there exists a point $x_{n}\in\R$ such that 
\[
u_{n}(T-,x_{n})~\geq~\|u_{n-1}(T-,\cdot)\|_{{\bf L}^{\infty}(\R)}+\theta_n.
\]
Set $ I_n\doteq \big[x_n-T2^{-n}\Phi_f(\theta_n/2),x_n\big]$. By  (\ref{O}) and (\ref{Phi}), we derive
\[
u_n(T-,x_{n})-u_n(T-,x)~\leq~\Phi_f^{-1}\left({2^n\over T}\cdot (x_{n}-x)\right)~\leq~{\theta_n\over 2}\quad\forall  x\in I_n,
\]
and 
\bel{L1-b1}
\begin{split}
\int_{I_n}u_n(T-,x)-u_{n-1}&(T-,x)dx~\geq~\int_{I_n}u_n(T-,x_{n})-{\theta_n\over 2}-\|u_{n-1}(T-,\cdot)\|_{{\bf L}^{\infty}(\R)}dx\\
&~\geq~\int_{I_n}{\theta_n\over 2}dx~=~T2^{-n-1}\theta_n\Phi_f(\theta_n/2).
\end{split}
\eeq
Let $t\mapsto \xi_1(t)$ and  $t\mapsto \xi_2(t)$ be the minimal  backward characteristics of $u_{n}$ and $u_{n-1}$ starting from  $\big(x_n-T2^{-n}\Phi_f(\theta_n/2),T\big)$ and $(x_n,T)$, respectively, which are defined as in Definition \ref{Char}. By Lemma \ref{hc-bb-u}, (\ref{T-un}), and the assumption {\bf (F)} on $f$, we estimate  
\bel{Car1}
\begin{split}
\sup_{t\in [\tau_{n-1},T]}|\xi_1(t)-\xi_2(t)|&~\leq~\sup_{t\in [\tau_{n-1},T]}\big\{|I_n|+|\xi_1(t)-\xi_1(T)|+|\xi_2(t)-\xi_2(T)|\big\}\\[2mm]
&~=~T2^{-n}\Phi_f(\theta_n/2)+2^{1+(1-n)(1-\gamma_p)}\Gamma_1T^{1-\gamma_p}\doteq~\Lambda_n.
\end{split}
\eeq
Consider the domain 
\[
\Omega_{0}~=~\big\{(x,t)\in \R\times [\tau_{n},T]: \xi_1(t)\leq x\leq \xi_2(t)\big\}.
\]
By the divergence theorem and the convexity of $f$, we estimate  
\[
\begin{split}
0&~=~\int_{\Omega_{0}}\Big(\mathrm{div}\big(f(u_{n}),u_n\big)-\mathrm{div}\big(f(u_{n-1}),u_{n-1}\big)\Big)dxdt\\
&~=~\int_{I_n}\Big(u_n(T-,x)-u_{n-1}(T-,x)\Big)dx-\int_{\tau_{n}}^{T}\big[f(u_n) -f(u_{n-1})- f'(u_{n}) (u_n-u_{n-1})\big](\xi_1(t),t)dt\\
&\qquad\qquad~+~\int_{\tau_{n}}^{T}\big[f(u_n)-f(u_{n-1})-f'(u_{n-1})(u_n-u_{n-1})\big](\xi_2(t),t)dt \\
&\qquad\qquad~-\int^{\xi_2(\tau_n)}_{\xi_1(\tau_n)}\Big(u_n(\tau_n+,x)-u_{n-1}(\tau_n+,x)\Big)dx\\
&~\geq~\int_{I_n}\Big(u_n(T-,x)-u_{n-1}(T-,x)\Big)dx-\int^{\xi_2(\tau_n)}_{\xi_1(\tau_n)}\Big(u_n(\tau_n+,x)-u_{n-1}(\tau_n+,x)\Big)dx,
\end{split}
\]
and this implies  that 
\bel{un-11}
\int_{I_n}\Big(u_n(T-,x)-u_{n-1}(T-,x)\Big)dx~\leq~\int^{\xi_2(\tau_n)}_{\xi_1(\tau_n)}\Big(u_n(\tau_n+,x)-u_{n-1}(\tau_n+,x)\Big)dx.
\eeq
Similarly, for every $\tau_{n-1}\leq s<t\leq \tau_n$ such that $\big\{\ell\in \mathbb{N}: s<\ell2^{-\nu}<t\big\}=\emptyset$, 
it holds 
\[
J^-(t)~\doteq~\int^{\xi_2(t)}_{\xi_{1}(t)}\Big(u_n(t-,x)-u_{n-1}(t-,x)\Big)dx~\leq~\int_{\xi_1(s)}^{\xi_2(s)}\Big(u_n(s+,x)-u_{n-1}(s+,x)\Big)dx~\doteq~J^{+}(s).
\]
By  Lemma \ref{hc-bb-u} and (\ref{Car1}), for every $\tau_{n-1}\leq t_\ell=\ell 2^{-\nu}\leq \tau_n$ we estimate   
\[
\begin{split}
J^{+}(t_\ell)-J^{-}(t_\ell)&~\leq~2^{-\nu}\cdot\int^{\xi_2(t_{\ell})}_{\xi_1(t_{\ell})}\big|G\big[u^{\nu,+}(t_{\ell}-,\cdot)\big](x)\big|~dx\\[2mm]
&~\leq~2^{-\nu}\big\|G\big[u^{\nu,+}(t_\ell-,\cdot)\big]\big\|_{{\bf L}^2(\R)}\cdot|\xi_2(t_\ell)-\xi_1(t_\ell)|^{1/2}\\[2mm]
&~\leq~\left(\|{\bf G}\|_{\infty} e^{\|{\bf G}\|_{\infty}T}\|\bar{u}\|_{{\bf L}^2(\R)}\right)\cdot 2^{-\nu} \Lambda^{1/2}_n.
\end{split}
\]
Noticing that $u_n(\tau_{n-1}+,\cdot)=u_{n-1}(\tau_{n-1}+,\cdot)$, we derive 
\[
\begin{split}
J^{-}(\tau_n)&~=~\int^{\xi_2(\tau_n)}_{\xi_1(\tau_n)}\Big(u_n(\tau_n-,x)-u_{n-1}(\tau_n-,x)\Big)dx~\leq~ \sum_{\tau_{n-1}\leq t_\ell<\tau_n}\big[J^{+}(t_\ell)-J^{-}(t_\ell)\big]
\\[2mm]
&~\leq~ \sum_{\tau_{n-1}\leq t_\ell<\tau_n}\left(\|{\bf G}\|_{\infty} e^{\|{\bf G}\|_{\infty}T}\|\bar{u}\|_{{\bf L}^2(\R)}\right)\cdot 2^{-\nu} \Lambda^{1/2}_n\\[2mm]
&~\leq~\left(\|{\bf G}\|_{\infty} e^{\|{\bf G}\|_{\infty}T}\|\bar{u}\|_{{\bf L}^2(\R)}\right)\cdot T 2^{-n} \Lambda^{1/2}_n.
\end{split}
\]
By (\ref{un-11}) and (\ref{n0}), we get
\[
\begin{split}
\int_{I_n}\Big(u_n(T-,x)-u_{n-1}(T-,x)\Big)dx&~\leq~J^+(\tau_n)-J^-(\tau_n)+J^-(\tau_n)\\
&~\leq~\left(\|{\bf G}\|_{\infty} e^{\|{\bf G}\|_{\infty}T}\|\bar{u}\|_{{\bf L}^2(\R)}\right)\cdot \big(2^{-\nu}+T 2^{-n}\big) \Lambda^{1/2}_n\\[2mm]
&~\leq~\left(\|{\bf G}\|_{\infty} e^{\|{\bf G}\|_{\infty}T}\|\bar{u}\|_{{\bf L}^2(\R)}\right)\cdot T 2^{1-n} \Lambda^{1/2}_n,
\end{split}
\]
and (\ref{L1-b1}), (\ref{Car1}) yield 
\[
\begin{split}
\theta_n\Phi_f(\theta_n/2)&~\leq~4\left(\|{\bf G}\|_{\infty} e^{\|{\bf G}\|_{\infty}T}\|\bar{u}\|_{{\bf L}^2(\R)}\right)\cdot  \Lambda^{1/2}_n\\[2mm]
&~\leq~4\left(\|{\bf G}\|_{\infty} e^{\|{\bf G}\|_{\infty}}\|\bar{u}\|_{{\bf L}^2(\R)}\right)\cdot\left(2^{-n}\Phi_f(\theta_n/2)+2^{1+(1-n)(1-\gamma_p)}\Gamma_1\right)^{1/2}.
\end{split}
\]
Two cases are considered: 
\begin{itemize}
\item If $2^{-n}\Phi_f(\theta_n/2)\leq 2^{1+(1-n)(1-\gamma_p)}\Gamma_1$, then by (\ref{Phi}) we have
\[
C_f2^{-p_1}\theta_n^{1+p_1}~\leq~\theta_n\Phi_f(\theta_n/2)~\leq~4\left(\|{\bf G}\|_{\infty} e^{\|{\bf G}\|_{\infty}}\|\bar{u}\|_{{\bf L}^2(\R)}\right)\cdot 2^{1+(1-n)(1-\gamma_p)/2}\Gamma^{1/2}_1,
\]
and this yields
\[
\theta_n~\leq~ \Gamma_1'  2^{-{n(1-\gamma_p)\over 2(1+p_1)}},\qquad \Gamma'_1~=~ \bigg(2^{(1-\gamma_p)/2}\Gamma^{1/2}_1C^{-1}_f \|{\bf G}\|_{\infty} e^{\|{\bf G}\|_{\infty}}\|\bar{u}\|_{{\bf L}^2(\R)} \, 2^{3+p_1}\bigg)^{{1 \over {1+p_1}}}.
\]
\item Otherwise, if $2^{-n}\Phi_f(\theta_n/2)\geq 2^{1+(1-n)(1-\gamma_p)}\Gamma_1$, then 
\[
\theta_n\Phi_f(\theta_n/2)~\leq~4 \left(\|{\bf G}\|_{\infty} e^{\|{\bf G}\|_{\infty}}\|\bar{u}\|_{{\bf L}^2(\R)}\right)\cdot 2^{(1-n)/2} \Phi_f^{1/2}(\theta_n/2),
\]
and by (\ref{Phi}) we have
\[
C^{1/2}_f 2^{-p_1/2}\theta_n^{1+p_1/2}~\leq~ \theta_n\Phi_f^{1/2}(\theta_n/2)~\leq~ \|{\bf G}\|_{\infty} e^{\|{\bf G}\|_{\infty}}\|\bar{u}\|_{{\bf L}^2(\R)} \, 2^{2+(1-n)/2},
\]
which yields
\[
\theta_n~\leq~ \Gamma_1'' 2^{-{n\over (2+p_1)}},\qquad \Gamma_1''~=~ \bigg(2^{1/2}C_f^{-1/2} \|{\bf G}\|_{\infty} e^{\|{\bf G}\|_{\infty}}\|\bar{u}\|_{{\bf L}^2(\R)} \, 2^{2+p_1/2}\bigg)^{{1 \over {1+p_1/2}}}.
\]
\end{itemize}
In particular, since $\gamma_p\in (0,1)$,  it holds that $2^{-{n(1-\gamma_p)\over 2(1+p_1)}}\geq 2^{-{n\over (2+p_1)}}$, and for every $2\leq n\leq n_0$,
\bel{b-theta}
\theta_n~\leq~\widetilde{\Gamma}_1\cdot 2^{-{n(1-\gamma_p)\over 2(p_1+1)}},\qquad \Tilde{\Gamma}_1~=~\max\{\Gamma'_1,\Gamma''_1\}.
\eeq
Hence, from (\ref{sum1}) and (\ref{te-a}), we get
\[
\begin{split}
\|u^{\nu,+}(T-,\cdot)\|_{{\bf L}^{\infty}(\R)}&~\leq~4\cdot \left({e^{\|{\bf G}\|_{\infty}T}  \|\bar{u}\|^{2}_{{\bf L}^2(\R)}
\over C_f T/2}\right)^{{1\over p_1+2}}+\sum_{n=2}^{n_0}\left(\widetilde{\Gamma}_1 \cdot 2^{-{n(1-\gamma_p)\over 2(p_1+1)}}+2^{-n}\right)\\
&~\leq~\Gamma_2\cdot \left({1\over T^{1/(p_1+2)}}+1\right)\qquad\mathrm{for~some}~\Gamma_2>0.
\end{split}
\]
With the same argument, one can also show that 
\[
\|u^{\nu,-}(T-,\cdot)\|_{{\bf L}^{\infty}(\R)}~\leq~\Gamma_2\cdot \left({1\over T^{1/p_1+2}}+1\right),
\]
and this yields (\ref{L-infty}).
\endproof
\subsubsection {Tightness property} Toward to the existence of the limit of $u^{\nu}$ and its ${\bf L}^{2}$-continuity,  we shall establish a {\it Tightness Property} for the sequence of approximating solutions $\{u^{\nu}\}_{\nu\geq 1}$ by using the following property on ${\bf G}$: If $g\in {\bf L}^2(\R)$ has  support contained in the interval  $[-r,r]$, then  for every $\kappa>0$, it holds
\bel{ke3}
\big\|{\bf G}[g]\big\|_{{\bf L}^2(\R\backslash [-r-\kappa,r+\kappa])}~\leq~2C_K \left\|g\right\|_{{\bf L}^2(\R)} \cdot \sqrt{r\over \kappa}\qquad\forall \kappa>0.
\eeq
Indeed, for every $x\in \R\backslash [-r-\kappa,r+\kappa]$ , one has 
\[
\begin{split}
{\bf G}[g](x)&~=~\int_{\R}K(z)\cdot\chi_{\R\backslash [-\kappa,\kappa]}(z)\cdot g(x-z)dz~=~\left(\big[K\cdot \chi_{\R\backslash [-\kappa,\kappa]}\big]*g\right)(x),
\end{split}
\]
and ({\bf K2}) yields
\[
\begin{split}
\int_{\R\backslash [-r-\kappa,r+\kappa]}\big|{\bf G}[g](x)\big|^2dx&~=~\int_{\R\backslash [-r-\kappa,r+\kappa]}\left|\left(\big[K\cdot \chi_{\R\backslash [-\kappa,\kappa]}\big]*g\right)(x)\right|^2dx\\[2mm]
&~\leq~\left\|K\cdot \chi_{\R\backslash [-\kappa,\kappa]}\right\|^2_{{\bf L}^2(\R)}\cdot \left\|g\right\|^2_{{\bf L}^1(\R)}\\[2mm]
&~\leq~\left(C_K^2\cdot \int_{\R\backslash [-\kappa,\kappa]}{1\over x^2}~dx\right)\cdot \left(2r\left\|g\right\|^2_{{\bf L}^2(\R)}\right)~=~{4C_K^2r\over \kappa}\cdot \left\|g\right\|^2_{{\bf L}^2(\R)}.
\end{split}
\]

\begin{lemma}\label{TN1} For every $\ve>0$, there exists $R_{\ve}>0$ sufficiently large such that 
\bel{TN}
\int_{\R\backslash [-R_{\ve},R_{\ve}]} \big|u^{\nu}(t,x)\big|^2dx~\leq~\ve,\qquad \nu\geq 1, T\in (0,1].
\eeq
\end{lemma}
{\bf Proof.} For a fixed $\nu\geq 1$ and $T\in (0,1]$, let $R_i^+(t)$ and $R_i^-(t)$ be the maximal and minimal backward characteristics of the approximate solution $u^\nu$ through the points $(T, R_i)$ and $(T, -R_i)$, respectively. Here the sequence of radii $(R_i)_{i\geq 0}$ is inductively defined as:
\bel{R-i}
R_0~=~0,\qquad R_i-R_{i-1}~=~ \ds{2^{2i+2}C_K^2\over \|{\bf G}\|^2_{\infty} } \cdot \big(R_{i-1}+\Gamma_1\big)+ 2\Gamma_1,\qquad i\geq 1.
\eeq
For each $t\in [0,T]$ and $i\geq 1$, we define the spaces 
\[
H_0(t)~\doteq~ \Big\{u\in {\bf L}^2(\R): ~ supp(u) \subseteq [R_1^-(t),R_1^+(t)]\Big\}
\]
\[
 H_i(t)~ \doteq ~\Big\{u\in {\bf L}^2(\R): ~ supp(u) \subseteq [R_{i}^+(t), R_{i+1}^+(t)]\cup [R_{i+1}^-(t), R_i^-(t)]\Big\}, 
\]
and  call
\[
K_i^-(t)~\doteq~ H_0(t) \oplus \dots \oplus H_{i-1}(t), \qquad K_i^+(t) ~\doteq~ H_i(t) \oplus H_{i+1}(t) \oplus \dots
\]
so that $K_i^-(t)\oplus K_i^+(t)={\bf L}^2(\R)$, with perpendicular projections $\pi_{i}^{\pm}(t):{\bf L}^2(\R)\to K_i^\pm(t)$ satisfying 
\bel{Proj}
\pi_i^-(t)(u)+\pi_i^+(t)(u)~=~u,\qquad \pi_i^-(t)(u)~=~\begin{cases} u(x), \quad x\in [R_i^-(t), R_i^+(t)], \\[3mm] 0, \qquad \,\, \text{else}. \end{cases}
\eeq
 Since the curves $R_i^-, R_i^+$ are characteristics, on each time subinterval $[t_{\ell-1},t_\ell)$ with $t_\ell=\ell 2^{-\nu}$ we have, 
\[
 \frac{d}{dt} \int_{\R} \big(\pi_i^+ \, u^{\nu}(t,\cdot) \big)^2 \, dx~=~ \frac{d}{dt} \int_{\R\setminus [R_i^-(t), R_i^+(t)]} \big(u^\nu(t,x)\big)^2 \, dx ~\leq~0.
\]
Hence the non-decreasing map $t\mapsto p_i(t)$ defined by 
\bel{pi}
p_i(t)~=~ \sup_{s\in [0,t]}\|\pi^{+}_i(s)\big(u^{\nu}(s,\cdot)\big)\|_{\bf L^2(\R)},\qquad t\in [0,T],
\eeq
satisfies 
\bel{p-w}
p_i(t)~=~p_i(t_{\ell-1}),\qquad t\in [t_{\ell-1},t_{\ell}[.
\eeq
Using {\bf (K1)}, we estimate 
\bel{p1}
\begin{split}
p_{1}(t_\ell)-p_{1}(t_{\ell-1})&~=~p_1(t_\ell)-p_1(t_{\ell}-)\\[1mm]
&~\leq~\|u^\nu(t_\ell+,\cdot)\|_{{\bf L}^2(\R\setminus [R_1^-(t), R_1^+(t)])}-\|u^\nu(t_\ell-,\cdot)\|_{{\bf L}^2(\R\setminus [R_1^-(t), R_1^+(t)])}\\[1mm]
&~\leq~2^{-\nu} \, \big\|{\bf G}[u^\nu(t_\ell-,\cdot)]\big\|_{{\bf L}^2(\R\setminus [R_1^-(t), R_1^+(t)])}~\leq~2^{-\nu}\|{\bf G}\|_{\infty}e^{\|{\bf G}\|_{\infty}} \|\bar{u}\|_{{\bf L}^2(\R)},
\end{split}
\eeq
and for all $i\geq 2$
\[
\begin{split}
p_{i}(t_\ell)-p_{i}(t_{\ell-1})&~=~p_i(t_\ell)-p_i(t_{\ell}-)~\leq~2^{-\nu} \, \|{\bf G}[u^\nu(t_\ell-,\cdot)]\|_{{\bf L}^2(\R\setminus [R_i^-(t), R_i^+(t)])} \\[1mm]
&~\leq~2^{-\nu} \, \Big(\| {\bf G} \big[ \pi_{i-1}^+ u^\nu(t_\ell-,\cdot) \big] \|_{{\bf L}^2(\R\setminus [R_i^-(t), R_i^+(t)])} \\[1mm]
&\qquad\qquad\qquad\qquad\qquad\qquad + \| {\bf G} \big[ \pi_{i-1}^- u^\nu(t_\ell-,\cdot) \big] \|_{{\bf L}^2(\R\setminus [R_i^-(t), R_i^+(t)])}\Big)\\[1mm]
&~\leq~2^{-\nu}\left(\|{\bf G}\|_{\infty} \cdot p_{i-1}(t_{\ell}-)+\| {\bf G} \big[ \pi_{i-1}^- u^\nu(t_\ell-,\cdot) \big] \|_{{\bf L}^2(\R\setminus [R_i^-(t), R_i^+(t)])}\right)\\[1mm]
&~\leq~2^{-\nu}\left(\|{\bf G}\|_{\infty} \cdot p_{i-1}(t_{\ell-1})+\| {\bf G} \big[ \pi_{i-1}^- u^\nu(t_\ell-,\cdot) \big] \|_{{\bf L}^2(\R\setminus [R_i^-(t), R_i^+(t)])}\right).
\end{split}
\]
By lemma  \ref{hc-bb-u}, we have 
\[
\big[-R_i+\Gamma_1, R_i-\Gamma_1\big]~\subseteq~\big[R^{-}_i(t),R^{+}_i(t)\big]~\subseteq~ \big[-R_i- \Gamma_1 , R_i+\Gamma_1 \big],
\]
and derive from (\ref{ke3}), (\ref{R-i}), and (\ref{L2-u}) that
\[
\begin{split}
\big\| {\bf G} \big[ \pi_{i-1}^- u^\nu(t_\ell-,\cdot) \big] \big\|_{{\bf L}^2(\R\setminus [R_i^-(t), R_i^+(t)])}&~\leq~ 2C_K\sqrt{R_{i-1}+ \Gamma_1 \over R_i-R_{i-1}-2\Gamma_1 }\cdot\| \pi_{i-1}^- u^\nu(t_\ell-,\cdot)\|_{{\bf L}^2(\R)}\\[2mm]
&~\leq~ 2^{-i} \, \|{\bf G}\|_{\infty} e^{\| {\bf G}\|_{\infty}} \cdot \|\bar{u}\|_{{\bf L}^2(\R)}.
\end{split}
\]
Thus, we get 
\bel{in-p}
p_{i}(t_\ell)-p_{i}(t_{\ell-1})~\leq~
2^{-\nu}\|{\bf G}\|_{\infty}\cdot  \left(p_{i-1}(t_{\ell-1})+ 2^{-i}\cdot e^{\| {\bf G}\|_{\infty}} \cdot \|\bar{u}\|_{{\bf L}^2(\R)}\right),\qquad i\geq 2.
\eeq
In particular, setting $s_\ell\doteq \ds\sum_{i=1}^{\infty} p_{i}(t_\ell)$, we derive from (\ref{p1}) and (\ref{in-p}) that  
\[
\begin{cases}
s_{\ell}-s_{\ell-1}~\leq~2^{-\nu}\|{\bf G}\|_{\infty}\cdot\left(s_{\ell-1}+2e^{\| {\bf G}\|_{\infty}} \cdot \|\bar{u}\|_{{\bf L}^2(\R)}\right),\\[2mm]
s_0~=~\ds\sum_{i=1}^{\infty}\|\pi^{+}_i(0)\big(u^{\nu}(0,\cdot)\big)\|_{\bf L^2(\R)}~=~\|\bar{u}\|_{{\bf L}^2(\R)}.
\end{cases}
\]
By a standard computation, we obtain 
\[
\begin{split}
s_{\ell}&~\leq~(1+2^{-\nu}\|{\bf G}\|_{\infty})^{\ell}\cdot\left(1+2e^{\| {\bf G}\|_{\infty}} \right)\cdot \|\bar{u}\|_{{\bf L}^2(\R)}~\leq~e^{t_\ell\|{\bf G}\|_{\infty}}\cdot\left(1+2e^{\| {\bf G}\|_{\infty}} \right)\cdot \|\bar{u}\|_{{\bf L}^2(\R)},
\end{split}
\]
and (\ref{p-w})  yields  
\[
\sum_{i=1}^{\infty}\sup_{s\in [0,T]}\|\pi^{+}_i(t)\big(u^{\nu}(t,\cdot)\big)\|_{\bf L^2(\R)}~=~\sum_{i=1}^{\infty}p_i(T)~\leq~e^{\|{\bf G}\|_{\infty}}\cdot\left(1+2e^{\| {\bf G}\|_{\infty}} \right)\cdot \|\bar{u}\|_{{\bf L}^2(\R)}.
\]
Hence, $\ds\lim_{i\to\infty}\sup_{s\in [0,T]}\|\pi^{+}_i(t)\big(u^{\nu}(t,\cdot)\big)\|_{\bf L^2(\R)}=0$ and (\ref{TN}) holds for every $\ve>0$.
\endproof

\subsection{Proof of Theorem \ref{Main1}}
After introducing the sequence of approximate solutions $(u^{\nu})_{\nu\geq 1}$ and deriving some a priori estimates along with the tightness property in the previous section, we now establish the existence of a weak entropy solution to (\ref{BL})-(\ref{G}) on $[0,1]\times \R$ with initial data ${\bar u}\in {\bf L}^2(\R)$.  The solution can then be prolonged to any time interval $[0,T]$ for $T>0$. The proof is divided into the following steps:
\medskip

{\bf 1. ${\bf L}^1$-estimates of approximate solutions.} For every $\delta>0$,  consider a new sequence of approximate solutions  $u^{\nu}_{\delta}:[\delta,1]\times\R\to\R$ which is defined as follows: for every $t\in [{\delta},1[$ such that  $t=(1-\theta_t)\cdot t_{\ell}+\theta_t\cdot t_{\ell+1}\in [\delta, \infty[$ for  $\theta_t\in [0,1[$ and $t_\ell=\ell2^{-\nu}$ with $\ell\in \{0,\cdots, 2^{\nu}-1\}$, we define
\bel{u-d}
u^{\nu}_{\delta}(t,x)~=~(1-\theta_t)\cdot S_{\delta}\big(u^{\nu}([t_\ell-\delta]+,\cdot)\big)(x)+\theta_t\cdot S_{\delta}\big(u^{\nu}([t_{\ell+1}-\delta]-,\cdot)\big)(x).
\eeq
Fix any $R>0$. For every $t\in [\delta,1]$, it holds
 \bel{d-12}
  \begin{split}
  \big\|u^{\nu}_{\delta}(t,\cdot)-u^{\nu}(t,\cdot)\big\|_{{\bf L}^1([-R,R])}&~\leq~(1-\theta_t)\cdot \big\|S_{\delta}\big(u^{\nu}([t_\ell-\delta]+,\cdot)\big)-u^{\nu}(t,\cdot)\big\|_{{\bf L}^1([-R,R])}\\[2mm]
 &\qquad\quad+\theta_t\cdot\|S_{\delta}\big(u^{\nu}([t_{\ell+1}-\delta]-,\cdot)-u^{\nu}(t,\cdot)\big)\|_{{\bf L}^1([-R,R])}.
  \end{split}
 \eeq
Let $s\mapsto x^{\nu,-}(s)$ and $s\mapsto x^{\nu,+}(s)$  be the minimal and maximal backward characteristics for $u^{\nu}$ through $(t,-R)$ and $(t,R)$,  respectively. For every subinterval, $[t_{\ell},t_{\ell+1}]\subset [t-\delta,t]$, we have
\[
\begin{split}
\int_{x^{\nu,-}(t_{\ell+1})}^{x^{\nu,+}(t_{\ell+1})}&\Big|u^{\nu}(t_{\ell+1}-,x)-S_{\delta-(t-t_{\ell+1})}u^{\nu}(t-\delta,x)\Big|~dx\\
&=~\int_{x^{\nu,-}(t_{\ell+1})}^{x^{\nu,+}(t_{\ell+1})}\left|S_{2^{-\nu}}u^{\nu}(t_{\ell}+,x)-S_{2^{-\nu}}\circ S_{\delta-(t-t_{\ell})}u^{\nu}(t-\delta,x)\right|dx\\
&\leq~\int_{x^{\nu,-}(t_{\ell})}^{x^{\nu,+}(t_{\ell})}\left|u^{\nu}(t_{\ell}+,x)-S_{\delta-(t-t_{\ell})}u^{\nu}(t-\delta,x)\right|dx\\
&\leq~\int_{x^{\nu,-}(t_{\ell})}^{x^{\nu,+}(t_{\ell})}\bigg(\left|u^{\nu}(t_{\ell}-,x)-S_{\delta-(t-t_{\ell})}u^{\nu}(t-\delta,x)\right|+2^{-\nu}\cdot |{\bf G}[u^{\nu}(t_{\ell}-,\cdot)](x)|\bigg)dx.
\end{split}
\] 
From (\ref{C-G}), (\ref{L2-u}), (\ref{Holder}), we derive 
\bel{CR0}
\int_{x^{\nu,-}(t_{\ell})}^{x^{\nu,+}(t_{\ell})} |{\bf G}[u^{\nu}(t_{\ell}-,\cdot)](x)|dx~\leq~\|{\bf G}\|_{\infty} e^{\| {\bf G}\|_{\infty}} \cdot \|\bar{u}\|_{{\bf L}^2(\R)} \cdot (R+\Gamma_1)^{1/2}~\doteq~C_{[0,R]},
\eeq
and get
\[
\begin{split}
\int_{x^{\nu,-}(t_{\ell+1})}^{x^{\nu,+}(t_{\ell+1})}\Big|u^{\nu}(t_{\ell+1}-,x)&-S_{\delta-(t-t_{\ell+1})}u^{\nu}(t-\delta,x)\Big|~dx\\
&\leq~\int_{x^{\nu,-}(t_{\ell})}^{x^{\nu,+}(t_{\ell})}\left|u^{\nu}(t_{\ell}-,x)-S_{\delta-(t-t_{\ell})}u^{\nu}(t-\delta,x)\right|+2^{-\nu}\cdot C_{[0,R]}.
\end{split}
\]
Hence, from (\ref{d-12}), an inductive argument yields
\bel{dist-1}
\begin{split}
 \big\|u^{\nu}_{\delta}(t,\cdot)-u^{\nu}(t,\cdot)\big\|_{{\bf L}^1([-R,R])}&\leq~(1-\theta_t)\cdot \sum_{t-\delta< t_i\leq t}2^{-\nu}\cdot C_{[0,R]}+\theta_t\cdot\sum_{t-\delta< t_i\leq t}2^{-\nu}\cdot C_{[0,R]}\\
 &\leq~C_{[0,R]}\cdot \delta\qquad\forall t\in [2\delta, 1].
 \end{split}
\eeq
Next, we prove the Lipschitz continuity of the map $t\mapsto u^{\nu}_{\delta}(t,\cdot)$ in ${\bf L}^1([-R,R])$. For every $s,t\in [t_\ell,t_{\ell+1}[\subseteq [\delta,1]$, it holds
\bel{L1}
\begin{split}
\|u^{\nu}_{\delta}(t,\cdot&)-u^{\nu}_{\delta}(s,\cdot)\|_{{\bf L}^1([-R,R])}\\[2mm]
&~=~\left\|S_{\delta}\big(u^{\nu}([t_{\ell+1}-\delta]-,\cdot)\big)-S_{\delta}\big(u^{\nu}([t_\ell-\delta]+,\cdot)\big)\right\|_{{\bf L}^1([-R,R])}\cdot |\theta_t-\theta_s|\\[1mm]
&~\leq~2^{\nu}\cdot \left\|S_{\delta}\big(u^{\nu}([t_{\ell+1}-\delta]-,\cdot)\big)-S_{\delta}\big(u^{\nu}([t_\ell-\delta]+,\cdot)\big)\right\|_{{\bf L}^1([-R,R])}\cdot |t-s|\\[1mm]
&~=~2^{\nu}\cdot \left\|S_{\delta}\circ S_{2^{-\nu}}\big(u^{\nu}([t_{\ell}-\delta]+,\cdot)\big)-S_{\delta}\big(u^{\nu}([t_\ell-\delta]+,\cdot)\big)\right\|_{{\bf L}^1([-R,R])}\cdot |t-s|\\[1mm]
&~=~2^{\nu}\cdot \left\|S_{2^{-\nu}}\circ S_{\delta}\big(u^{\nu}([t_{\ell}-\delta]+,\cdot)\big)-S_{\delta}\big(u^{\nu}([t_\ell-\delta]+,\cdot)\big)\right\|_{{\bf L}^1([-R,R])}\cdot |t-s|.
\end{split}
\eeq
From Lemma \ref{F-BV}, we have 
\bel{L-b}
\big\|S_{\delta}\big(u^{\nu}([t_\ell-\delta]+,\cdot)\big)\big\|_{{\bf L}^{\infty}(\R)}, \left\|u^{\nu}_{\delta}(t,\cdot)\right\|_{{\bf L}^{\infty}(\R)}~\leq~M_{\delta}~=~4\cdot\left({\|\bar{u}\|^2_{{\bf L}^2(\R)}\over C_f\delta}\right)^{1/(p_1+2)}
\eeq
for all $t\in [\delta,1]$, and
\bel{TV-b1}
\begin{split}
TV^{1/p_1}\big\{u^{\nu}_{\delta}(t,\cdot);[-R,R]\big\}&~\leq~(1-\theta_t)\cdot TV^{1/p_1}\big\{S_{\delta}\big(u^{\nu}([t_\ell-\delta]+,\cdot)\big);[-R,R]\big\}\\[1mm]
&\qquad\qquad\qquad +\theta_t\cdot TV^{1/p_1}\big\{S_{\delta}\big(u^{\nu}([t_{\ell+1}-\delta]+,\cdot)\big);[-R,R]\big\}\\[1mm]
&~\leq~{4R\over C_f\delta}+{2\over C_f}\cdot \|f'\|_{{\bf L}^{\infty}([-M_{\delta},M_{\delta}])}.
\end{split}
\eeq
Set $C_{[1,\delta]}\doteq \sup\{|f'(s)|:|s|\leq  M_{\delta}\}$. By (\ref{O}), we estimate 
\[
\begin{split}
\tv\left\{f'\Big(S_{\delta}\big(u^{\nu}([t_\ell-\delta]+,\cdot)\big)\Big); [-R,R]\right\}&~\leq~2C_{[1,\delta]}+{4R\over \delta} 
\end{split}
\]
and {\bf (F)} yields
\bel{TV-fu}
\begin{split}
\tv&\big\{f\Big(S_{\delta}\big(u^{\nu}(t_\ell-\delta,\cdot)\big)\Big); [-R,R]\big\}\\[1mm]
&~\leq~\sup_{|u|,|v|\leq M_{\delta}, u\neq v}{|f(u)-f(v)|\over |f'(u)-f'(v)|}\cdot \tv\left\{f'\Big(S_{\delta}\big(u^{\nu}(t_\ell-\delta,\cdot)\big)\Big); [-R,R]\right\}\\[1mm]
&~\leq~\Lambda_{M_{\delta}}\cdot \left(2C_{[1,\delta]}+{4R\over \delta}\right)~\doteq~C_{[1,\delta,R]}.
\end{split}
\eeq
Hence,  we get
\[
\begin{split}
\big\|S_{2^{-\nu}}\circ S_{\delta}\big(u^{\nu}([t_{\ell}-\delta]+,\cdot)\big)&-S_{\delta}\big(u^{\nu}([t_\ell-\delta]+,\cdot)\big)\big\|_{{\bf L}^1([-R,R])}~\leq~2^{-\nu}\cdot C_{[1,\delta,R]},
\end{split}
\]
and (\ref{L1}) yields 
\bel{es1}
\big\|u^{\nu}_{\delta}(t,\cdot)-u^{\nu}_{\delta}(s,\cdot)\big\|_{{\bf L}^1([-R,R])}~\leq~C_{[1,\delta,R]}\cdot |t-s|,\qquad s,t\in [\delta,1].
\eeq

\n {\bf 2. Existence of a limiting function.} By (\ref{TV-b1}) and the compactness property of $BV^{1/p_1}$ functions in ${\bf L}^1$, one can construct a subsequence $\{u^{\nu}_{\delta}\}_{\nu\in\mathcal{I}_{\delta}}$ such that $u^{\nu}_{\delta}(t,\cdot)$ converges to $u_{\delta}(t,\cdot)$ pointwise in $[-R,R]$ and hence  also in ${\bf L}^1([-R,R],\R)$ at every rational $t\in [0,1]$. By the Lipschitz continuity in (\ref{es1}), this can be uniquely extended to a map $t\mapsto u_{\delta}(t,\cdot)$ in ${\bf L}^1([-R,R],\R)$ such that 
\bel{cv1}
\lim_{\mathcal{I}_{\delta}\ni\nu\to\infty} \|u^{\nu}_{\delta}(t,\cdot)-u_{\delta}(t,\cdot)\|_{{\bf L}^1([-R,R])}~=~0,\qquad t\in [2\delta,1],
\eeq
\bel{cv2}
\lim_{\mathcal{I}_{\delta}\ni\nu\to\infty}\|u^{\nu}_{\delta}-u_{\delta}\|_{{\bf L}^1([2\delta,1]\times [-R,R])}~=~0.
\eeq
Since this construction can be repeated for every integer $R>0$ and  $\delta=2^{-\mu}$ with $\mu\in \mathbb{Z}^+$, we can select a countable set of indices $\mathcal{I}\subseteq \mathbb{N}$ such that both (\ref{cv1}) and (\ref{cv2}) hold for $\nu\in \mathcal{I}$. 

In this case, from (\ref{dist-1}) and (\ref{cv1}), for any integer $R>0$ and  $\mu\in \mathbb{Z}^+$, it holds
\[
\begin{split}
\limsup_{\nu,\nu'\in\mathcal{I}} \big\|u^{\nu}(t,\cdot)&-u^{\nu'}(t,\cdot)\big\|_{{\bf L}^1([-R,R])}~\leq~\limsup_{\nu\to\infty}\big\|u^{\nu}(t,\cdot)-u^{\nu}_{2^{-\mu}}(t,\cdot)\big\|_{{\bf L}^1([-R,R])}\\[2mm]
&+\limsup_{\nu\to\infty}\big\|u^{\nu}_{2^{-\mu}}(t,\cdot)-u^{\nu'}_{2^{-\mu}}(t,\cdot)\big\|_{{\bf L}^1([-R,R])}+\limsup_{\nu\to\infty}\big\|u^{\nu'}_{2^{-\mu}}(t,\cdot)-u^{\nu'}(t,\cdot)\big\|_{{\bf L}^1([-R,R])}\\[2mm]
&\leq~2\cdot C_{[0,R]}\cdot 2^{-\mu}.
\end{split}
\]

Hence,  the sequence $\{u^{\nu}(t,\cdot)\}_{\nu\in \mathcal{I}}$ is Cauchy  in ${\bf L}^1([-R,R])$ and for some limit function $u(t,x)$ such that
\bel{l1}
\lim_{\mathcal{I}\ni\nu\to\infty} \|u^{\nu}(t,\cdot)-u(t,\cdot)\|_{{\bf L}^1([-R,R])}~=~0,\qquad t\in [2^{-\mu},1],
\eeq
\bel{l2}
\lim_{\mathcal{I}\ni\nu\to\infty}\|u^{\nu}(t,\cdot)-u(t,\cdot)\|_{{\bf L}^1([2^{-\mu},1]\times [-R,R])}~=~0.
\eeq
In particular, $u^{\nu}(t,x)$ converges to $u(t,x )$ for a.e. $x\in \R$ and $t\in (0,1]$. From  (\ref{e-s-l2}), Lemma \ref{L-infty-b}, and (\ref{dist-1}), we also derive
\bel{u-l2-b}
\|u(t,\cdot)\|_{{\bf L}^2(\R)}~\leq~\liminf_{\nu\to\infty}\|u^{\nu}(t,\cdot)\|_{{\bf L}^2(\R)}~\leq~e^{\| {\bf G}\|_{\infty}t} \cdot\|\bar{u}\|_{{\bf L}^2(\R)},
\eeq
and 
\bel{ze-c}
\begin{split}
\lim_{\delta\to 0+} \|u(\delta,\cdot)-\bar{u}\|_{{\bf L}^1([-R,R])}&~=~\lim_{\delta\to 0+} \|u(\delta,\cdot)-S_\delta(\bar{u})\|_{{\bf L}^1([-R,R])}\\
&~=~\lim_{\delta\to 0+}\left(\lim_{\nu\to\infty} \left(\|u^{\nu}(\delta,\cdot)-u^{\nu}_{\delta}(\delta,\cdot)\|_{{\bf L}^1([-R,R])}\right)\right)~=~0,
\end{split}
\eeq
and 
\bel{L-b-u}
\|u(t-,\cdot)\|_{{\bf L}^{\infty}(\R)}~\leq~\Gamma_2\cdot \left({1\over t^{1/(2+p_1)}}+1\right)\quad\forall t\in \big\{\ell 2^{-\nu}: 1\leq \ell\leq 2^{\nu},\nu\geq 1\big\}.
\eeq

{\bf 3. The map $t\mapsto u(t,\cdot)$ is continuous from $[0,1]\to {\bf L}^2(\R)$.} Fix $\bar{t}\in (0,1]$ and $\delta_0>0$ sufficiently small such that $[\bar{t}-\delta_0,\bar{t}+\delta_0]\subset (0,1]$ . For every $t,s\in [\bar{t}-\delta_0,\bar{t}+\delta_0]$ with $t>s$, it follows from (\ref{dist-1}) and (\ref{es1}) that 
\[
\begin{split}
 \big\|u^{\nu}(t,\cdot)-u^{\nu}(s,\cdot)&\big\|_{{\bf L}^1([-R,R])}~\leq~ \big\|u^{\nu}(t,\cdot)-u^{\nu}_{\delta}(t,\cdot)\big\|_{{\bf L}^1([-R,R])}+\big\|u^{\nu}(s,\cdot)-u^{\nu}_{\delta'}(s,\cdot)\big\|_{{\bf L}^1([-R,R])}\\[2mm]
 &\qquad\qquad\qquad+\big\|u_{\delta}^{\nu}(t,\cdot)-u^{\nu}_{\delta}(s,\cdot)\big\|_{{\bf L}^1([-R,R])}+\big\|u_{\delta}^{\nu}(s,\cdot)-u^{\nu}_{\delta'}(s,\cdot)\big\|_{{\bf L}^1([-R,R])}\\[2mm]
 &~\leq~C_{[0,R]}\cdot(\delta+\delta')+C_{[1,\delta,R]}\cdot |t-s|+\big\|u_{\delta}^{\nu}(s,\cdot)-u^{\nu}_{\delta'}(s,\cdot)\big\|_{{\bf L}^1([-R,R])}.
\end{split}
\]
Choosing $\delta=\delta'+(t-s)$, we have
\[
\begin{split}
\big\|u_{\delta}^{\nu}(s,\cdot)&-u^{\nu}_{\delta'}(s,\cdot)\big\|_{{\bf L}^1([-R,R])}\\[2mm]
&\qquad~\leq~(1-\theta_s)\cdot \left\|S_{(t-s)}\circ S_{\delta'}\big(u^{\nu}([s_k-\delta']+,\cdot)\big)-S_{\delta'}\big(u^{\nu}([s_{k}-\delta']+,\cdot)\big)\right\|_{{\bf L}^1([-R,R])}\\[2mm]
&\qquad\qquad+\theta_s\cdot \left\|S_{(t-s)}\circ S_{\delta'}\big(u^{\nu}([s_{k+1}-\delta']-,\cdot)\big)-S_{\delta'}\big(u^{\nu}([s_{k+1}-\delta']-,\cdot)\big)\right\|_{{\bf L}^1([-R,R])}\\[2mm]
&\qquad~\leq~2(t-s)\cdot C_{[1,\delta',R]}
\end{split}
\]
 with $s\in [s_k,s_{k+1}[\subset [\bar{t}-\delta_0,\bar{t}+\delta_0]\subset (0,1]$.  This implies that
\bel{ke11}
\begin{split}
 \big\|u^{\nu}(t,\cdot)-u^{\nu}(s,\cdot)\big\|_{{\bf L}^1([-R,R])}&~\leq~C_{[0,R]}\cdot(\delta+\delta')+\big(C_{[1,\delta,R]}+2 C_{[1,\delta',R]}\big)\cdot(t-s)\\[2mm]
 &~\leq~2C_{[0,R]}\cdot \delta' +\big(3C_{[1,\delta',R]}+C_{[0,R]}\big)\cdot (t-s).
 \end{split}
\eeq
Notice that the map $\delta'\mapsto K(\delta')$, defined by
\bel{K}
K(\delta')~=~3C_{[1,\delta',R]}+C_{[0,R]}~=~3\cdot \Lambda_{M_{\delta}}\cdot \left(2C_{[1,\delta]}+{4R\over \delta}\right)+C_{[0,R]}
\eeq
is strictly decreasing on $(0,\infty)$ with $\ds\lim_{\delta'\to 0+}K(\delta')=\infty$. In particular, consider the modulus of continuity   $\omega:[0,\infty)\to [0,\infty)$ such that 
\bel{mod}
\omega(0)~=~0,\qquad  \omega(\tau)~=~K^{-1}(\tau^{-1/2})\qquad\forall \tau>0.
\eeq
By choosing $\delta'=\omega(t-s)$ in (\ref{ke11}), we obtain 
\[
\big\|u^{\nu}(t,\cdot)-u^{\nu}(s,\cdot)\big\|_{{\bf L}^1([-R,R])}~\leq~2C_{[0,R]}\cdot \omega(t-s)+ (t-s)^{1/2},
\]
and (\ref{l1}) yields the H\"older continuity estimate for $u$ in ${\bf L}^1([-R,R])$
\bel{Holder-1}
\big\|u(t,\cdot)-u(s,\cdot)\big\|_{{\bf L}^1([-R,R])}~\leq~2C_{[0,R]}\cdot \omega(t-s)+ (t-s)^{1/2}.
\eeq
By (\ref{L-b-u}) and the density of the set $ \big\{\ell 2^{-\nu}: 1\leq \ell\leq 2^{\nu},\nu\geq 1\big\}$ in $[0,1]$, we get
\[
\|u(t,\cdot)\|_{{\bf L}^{\infty}(\R)}~\leq~\Gamma_{2}\cdot \left({1\over t^{1/(2+p_1)}}+1\right)\quad \forall t\in (0,1].
\]
%
%
%
Hence, for $ s,t\in [\bar{t}-\delta_0,\bar{t}+\delta_0]$, we have
\[
\begin{split}
\|u(t,\cdot)-u(s,\cdot)\|^2_{{\bf L}^2([-R,R])}&~=~ \big(\|u(t,\cdot)\|_{{\bf L}^{\infty}}+\|u(s,\cdot)\|_{{\bf L}^{\infty}(R)}\big)\cdot \|u(t,\cdot)-u(s,\cdot)\|_{{\bf L}^1([-R,R])}\\[2mm]
&~\leq~\Gamma_2 \cdot 2\cdot \left({1\over (\bar{t}-\delta_0)^{1/(2+p_1)}}+1\right)\cdot \left(2C_{[0,R]}\cdot \omega(t-s)+ (t-s)^{1/2}\right),
\end{split}
\]
which implies that 
\[
\lim_{t\to \bar{t}} \|u(t,\cdot)-u(\bar{t},\cdot)\|^2_{{\bf L}^2([-R,R])}~=~0,\qquad\forall R>0.
\]
By  Lemma \ref{TN1}, for every $\ve>0$, we have 
 \[
 \begin{split}
  \|u(t,\cdot)\|_{{\bf L}^2(\R\backslash[-R_{\ve},R_{\ve}])}&~\leq~\liminf_{\mathcal{I}\ni\nu\to\infty}\|u^{\nu}(t,\cdot)\|_{{\bf L}^2(\R\backslash[-R_{\ve},R_{\ve}])}~\leq~\ve\quad\forall t\in (0,1],
 \end{split}
 \]
and we derive
\[
 \lim_{t\to \bar{t}} \|u(t,\cdot)-u(\bar{t},\cdot)\|^2_{{\bf L}^2(\R)}~\leq~2\ve+\lim_{t\to \bar{t}} \|u(t,\cdot)-u(\bar{t},\cdot)\|^2_{{\bf L}^2([-R_{\ve},R_{\ve}])}~=~2\ve.
 \]
Since the above holds for arbitrary $\ve>0$, the map $t\mapsto u(t,\cdot)$ is continuous at  $\bar{t}\in (0,1]$. 

It remains to prove that this map is continuous at $t=0$. Consider any sequence $t_{n}\to 0+$. Thanks to the uniform bound on $\|u(t,\cdot)\|_{{\bf L}^2(\R)}$ in (\ref{e-s-l2}),  we can, by possibly taking a subsequence, assume the weak convergence $u(t_n,\cdot) \rightharpoonup w$ for some $w\in {\bf L}^2(\R)$. By (\ref{ze-c}) and (\ref{u-l2-b}), we get $w=\bar{u}$ and 
\[
 \limsup_{t_n\to 0+}\|u(t_n\cdot)\|_{{\bf L}^2(\R)}~\leq~\limsup_{t_n\to 0+} \left(e^{\| {\bf G}\|_{\infty}t_n} \cdot\|\bar{u}\|_{{\bf L}^2(\R)}\right)~=~\|\bar{u}\|_{{\bf L}^2(\R)}.
\]
This implies the strong convergence $u(t_n,\cdot)\rightarrow \bar{u}$ in ${\bf L}^2(\R)$ as $t_n\to 0+$.
\v

{\bf 4.  Entropy weak condition.} To complete the proof, we show that $u$ is an entropy weak solution of (\ref{BL})-(\ref{G}). For any given   $k\in \R$, consider a pair of convex entropy and flux 
\[
\eta(u)~=~|u-k|,\qquad q(u)~=~\mathrm{sign} (u-k)\cdot \left(f(u)-f(k)\right).
\]
To obtain (\ref{ei2}), we approximate $(\eta,q)$ by  a sequence of smooth entropies, say 
\[
\eta_n(u)~=~\sqrt{(u-k)^2+n^{-1}},\qquad q_n(u)~=~\int_{k}^{u}\eta'_n(s)f'(s)ds.
\]
For every $n\geq 1$ and $\phi\in\mathcal{C}_{c}^{\infty}(]0,\infty[\times\R)$, by (\ref{u-nu}), (\ref{L-infty}) and (\ref{L-b-u}), we have
\[
\begin{split}
\int\int\Big(\eta_n(u)\phi_t&+q_n(u)\phi_x\Big)dxdt~=~\lim_{\mathcal{I}\ni\nu\to\infty}\int\int\Big(\eta_n(u^{\nu})\phi_t+q_n(u^{\nu})\phi_x\Big)dxdt\\
&=~\lim_{\mathcal{I}\ni\nu\to\infty}\sum_{\ell=0}^{\infty}\int_{t_\ell}^{t_{\ell+1}}\int_{\R}\Big(\eta_n(u^{\nu})\phi_t+q_n(u^{\nu})\phi_x\Big)dxdt\\
&\geq~\lim_{\mathcal{I}\ni\nu\to\infty}\sum_{\ell=0}^{\infty}\int_{\R}\Big(\eta_n\big(u^{\nu}(t_{\ell+1}-,x)\big)\phi(t_{\ell+1},x)-\eta_n\big(u^{\nu}(t_{\ell}+,x)\big)\phi(t_\ell,x) \Big)dxdt\\
&=~\lim_{\mathcal{I}\ni\nu\to\infty}\sum_{\ell=1}^{\infty}\int_{\R}\big[\eta_n\big(u^{\nu}(t_{\ell}-,x)\big)-\eta_n\big(u^{\nu}(t_{\ell}+,x)\big)\big]\phi(t_\ell,x) dx\\
&=~-\lim_{\mathcal{I}\ni\nu\to\infty}\sum_{\ell=1}^{\infty}\int_{\R}\int_{t_{\ell}}^{t_{\ell+1}}\eta'_n\big(u^{\nu}(t_{\ell}-,x)\big)\cdot {\bf G}[u^{\nu}(t_\ell-,\cdot)](x)\cdot \phi(t_\ell,x)  dxdt\\
&=~-\int_{\R}\int_{0}^{\infty}\eta'_n\big(u(t,x)\big)\cdot {\bf G}[u(t,\cdot)](x)\cdot \phi(t,x) dxdt.
\end{split}
\]
Finally, taking $n\to\infty$, we obtain
\[
\begin{split}
\int\int \Big[|u-k|\phi_t+\mathrm{sign} (u-k)&\cdot \big(f(u)-f(k)\big)\phi_x\Big]dxdt~=~\lim_{n\to\infty}\int\int\Big(\eta_n(u)\phi_t+q_n(u)\phi_x\Big)dxdt\\
&=~-\lim_{n\to\infty} \int_{\R}\int_{0}^{\infty}\eta'_n\big(u(t,x)\big)\cdot {\bf G}[u(t,\cdot)](x)\cdot \phi(t,x) dt dx\\
&=~- \int_{\R}\int_{0}^{\infty} \mathrm{sign}\big(u(t,x)-k\big)\cdot {\bf G}[u(t,\cdot)](x)\cdot \phi(t,x) dt dx,
\end{split}
\]
which yields (\ref{ei2}).
\endproof
\section{Uniqueness results and fractional BV regularity}\label{unique}
In this section, we first establish a uniqueness result for bounded total variation entropy week solutions to a scalar balance law with nonlocal source (\ref{BL}) in the spatially periodic case. Furthermore, for ${\bf L}^1$-kernels $K$, we prove well-posedness of ${\bf L}^1(\R)$ entropy weak solutions $u$ to (\ref{BL})-(\ref{G}). For every $\tau>0$, the function $u(\tau,\cdot)$ belongs to a fractional BV space.
\setcounter{equation}{0}
\subsection{Periodic solutions with periodic sources}
For a given $P>0$, let ${\bf L}^2_{\mathrm{per}}(\R)$ denote  the space of $2P$-periodic functions which are square integrable over the interval $[-P,P]$, equipped with the norm
\[
\|g\|_{{\bf L}^2}~\doteq~\left(\int_{-P}^{P}|g(x)|^2dx\right)^{1/2},\qquad g\in {\bf L}^2_{\mathrm{per}}(\R).
\]
Throughout this subsection, we impose the following assumption on the kernel:
\begin{itemize}
\item [{\bf (K)}] The kernel  $K\in\mathcal{C}^1(\R\backslash 2P\mathbb{Z})$ is an odd $2P$-periodic function  and satisfies 
\bel{ap1}
\big|K^{(i)}(x)\big|~\leq~{C_K\over |x|^{i+1}},\qquad x\in [-P,P]\backslash\{0\}, ~ i = 0,1,
\eeq
for some constant $C_K>0$.
\end{itemize}
Under this assumption, we  can define the singular integral operator of periodic convolution type ${\bf G}_{\mathrm{per}}:{\bf L}_{\mathrm{per}}^2(\R)\to {\bf L}_{\mathrm{per}}^2(\R)$ by 
\[
{\bf G}_{\mathrm{per}}[g]~=~\lim_{\ve\to 0+}\int_{(-P,P)\backslash (-\ve,\ve)}K(y)\cdot g(x-y)~dy,\qquad g\in {\bf L}^2_{\mathrm{per}}(\R).
\]
%
On ${\bf L}^2_{\mathrm{per}}(\R)$, we consider the evolution equation 
\bel{EE}
u_t+f(u)_x~=~{\bf G}_{\mathrm{per}}[u],\qquad u(0,\cdot)~=~\bar{u}~\in~{\bf L}^2_{\mathrm{per}}(\R).
\eeq
\begin{definition} {\it By an {\bf entropy weak solution} of (\ref{EE})
we mean a function $u\in \L^1_{loc}([0,\infty[\,\times\R)$ with the following properties.
\begi
\item[(i)] The map $t\mapsto u(t,\cdot)$ is continuous with values in $\L^2_{\mathrm{per}}(\R)$ and
satisfies the initial condition in (\ref{EE}).
\item[(ii)]  For any $k\in \R$ and every nonnegative 
test function $\phi\in\C^1_c(]0,\infty[\,\times\R)$
one has
\bel{ei}\dint  \Big[ |u-k|\phi_t + \big(f(u)-f(k)\big)\hbox{\rm sign}(u-k) \phi_x
+{\bf G}_{\mathrm{per}}[u](x)\hbox{\rm sign}(u-k)\phi  \Big]\, dxdt~\geq~0.\eeq
\endi
}
\end{definition}
A global entropy weak solution of (\ref{EE}) can be obtained as the limit of the flux-splitting approximation described in Section \ref{sec:2}. In the periodic case, the tightness property of the sequence of approximating solutions is not required and  the analysis is considerably simpler, thanks to the continuous embedding ${\bf L}^2([-P,P])\hookrightarrow {\bf L}^1([-P,P])$.

Toward to the uniqueness of a solution to (\ref{EE}) within the class of  bounded total variation functions, we establish an ${\bf L}^1([-P,P])$-bound on ${\bf G}_{\mathrm{per}}[w]$ with $w$ being in a certain class of periodic functions with period $2P$.
\begin{lemma}\label{BV-L1} Let $w$ be  a periodic function with period $2P$ such that 
\[
\int_{-P}^{P}w(x)dx~=~0,\qquad \TV \{w; [-P,P]\}~=~\beta_w~<~\infty.
\]
Then it holds
\bel{L-e-p}
\big\|{\bf G}[w]\big\|_{{\bf L}^1([-P,P])}~\leq~C\big\|w\big\|_{{\bf L}^1([-P,P])}\left(2+3\ln 2+2\ln P+ 2\ln\beta_w-2\ln\big\|w\big\|_{{\bf L}^1([-P,P])}\right).
\eeq
\end{lemma}
{\bf Proof.} {\bf 1.} We first provide an ${\bf L}^1([-P,P])$-bound on ${\bf G}_{\mathrm{per}}\big[\chi_{[a,b]}\big]$ where $\chi_{[a,b]}$ denotes the characteristic function of the interval $[a,b]$, extended to the whole real line by $2P$ periodicity. For every $[a,b]\subseteq [-P,P]$, we prove that 
\bel{L-c}
\big\|{\bf G}_{\mathrm{per}}\big[\chi_{[a,b]}\big]\big\|_{{\bf L}^1([-P,P])}~\leq~C(b-a)\big(2+5\ln 2+2\ln P-2 \ln (b-a)\big|\big).
\eeq
 For every $x\in [-P,P]\backslash\{0\}$, we have 
\[
\begin{split}
{\bf G}_{\mathrm{per}}\big[\chi_{[a,b]}\big](x)&~=~\lim_{\ve\to 0+}\left(\int^{P}_{\ve}K(y)\chi_{[a,b]}(x-y)dy +\int_{-P}^{-\ve}K(y)\chi_{[a,b]}(x-y)dy\right)\\[2mm]
&~=~\lim_{\ve\to 0+}\left(\int^{P}_{\ve}K(y)\chi_{[a,b]}(x-y)dy-\int_{\ve}^{P}K(y)\chi_{[a,b]}(x+y)dy \right)\\[2mm]
&~=~\lim_{\ve\to 0+}\left(\int^{P}_{\ve}K(y)\, \left[\chi_{[x-b,x-a]}(y)-\chi_{[a-x,b-x]}(y)\right] dy \right).
\end{split}
\]
Assume that $[a,b]\subseteq [0,P]$. Then three cases are considered: 
\begin{itemize}
\item If $b<x<P$, then $[x-b,x-a]\subset (0,P)$ and 
\[
\begin{split}
\big|{\bf G}_{\mathrm{per}}\big[\chi_{[a,b]}\big](x)\big|~=~\left|\int_{x-b}^{x-a}K(y)dy\right|~\leq~C \int_{x-b}^{x-a}{1\over y}~dy~\leq~C \big(\ln (x-a)-\ln(x-b)\big).
\end{split}
\]
\item If $-P<x<a$, then $(a-x,b-x)\subset (0,P)$ and 
\[
\begin{split}
\big|{\bf G}_{\mathrm{per}}\big[\chi_{[a,b]}\big](x)\big|~=~\left|\int_{a-x}^{b-x}K(y)dy\right|~\leq~C \int_{a-x}^{b-x}{1\over y}~dy~\leq~C\cdot \big(\ln(b-x)-\ln(a-x)\big).
\end{split}
\]
\item Otherwise, if $x\in (a,b)$, then 
\[
\begin{split}
\big|{\bf G}_{\mathrm{per}}\big[\chi_{[a,b]}\big](x)\big|&~=~\lim_{\ve\to 0+}\left|\int_{\ve}^{x-a}K(y)dy-\int_{\ve}^{b-x}K(y)dy\right|~\leq~\left|\int_{x-a}^{b-x}|K(y)|~dy\right|\\[2mm]
&~\leq~C \left|\int_{x-a}^{b-x}{1\over y}~dy\right|~\leq~C \left|\ln(b-x)-\ln(x-a)\right|.
\end{split}
\]
\end{itemize}
By  a direct computation, we estimate
\[
\begin{split}
\left\|{\bf G}_{\mathrm{per}}\big[\chi_{[a,b]}\right\|_{{\bf L}^1[-P,P]}&~\leq~C \int_{-P}^{a}\Big(\ln(b-x)-\ln(a-x)\Big)dx\\
&\qquad+C\int_{a}^{b}\left|\ln(b-x)-\ln(x-a)\right|dx+C\int_{b}^{P}\Big(\ln (x-a)-\ln(x-b)\Big)dx\\
&~=~2C (b-a) \big[\ln2- \ln(b-a)\big]+C \big[(P+b)\ln(P+b)-(P+a)\ln (P+a)\big]\\[2mm]
&\qquad\qquad\qquad + C\big[(P-a)\ln(P-a)-(P-b)\ln(P-b)\big]\\[2mm]
&~\leq~2C (b-a) \big[\ln2-\ln(b-a)\big]+C (b-a)\big[1+\ln 2+\ln P\big]\\[2mm]
&\qquad\qquad\qquad +C(b-a)\big[1+\ln P\big]\\[2mm]
&~\leq~C(b-a)\big[2+3\ln 2+2\ln P-2\ln (b-a)\big]
\end{split}
\]
Similarly, we can  also prove that for every $[a,b]\subseteq [-P,0]$, it holds
\[
\left\|{\bf G}_{\mathrm{per}}\big[\chi_{[a,b]}\right\|_{{\bf L}^1[-P,P]}~\leq~C(b-a)\big[2+3\ln 2+2\ln P-2\ln (b-a)\big].
\]
Hence, for every $-P\leq a<0<b\leq P$, using the convexity of the map $x\mapsto x\ln x$, we estimate  
\[
\begin{split}
\big\|{\bf G}_{\mathrm{per}}\big[\chi_{[a,b]}\big]&\big\|_{{\bf L}^1[-P,P]}~\leq~\left\|{\bf G}_{\mathrm{per}}\big[\chi_{[a,0]}\big]\right\|_{{\bf L}^1[-P,P]}+\left\|{\bf G}_{\mathrm{per}}\big[\chi_{[0,b]}\big]\right\|_{{\bf L}^1[-P,P]}\\[2mm]
&~\leq~C |a| \big(2+3\ln 2+2\ln P-2\ln|a|\big)+C|b|\big(2+3\ln 2+2\ln P-2\ln|b|\big)\\[2mm]
&~\leq~ C(b-a)\big(2+3\ln 2+2\ln P\big)-2C (|a| \ln |a|+|b| \ln |b|)\\[2mm]
&~\leq~ C(b-a)\big(2+3\ln 2+2\ln P\big)-2C (b-a) \big(\ln (b-a)-\ln 2\big)\\[2mm]
&~\leq~C(b-a)\big(2+5\ln 2+2\ln P-2 \ln (b-a)\big),
\end{split}
\]
which proves (\ref{L-c}).
\medskip

{\bf 2.} For every $\ve>0$, we can approximate $w$  with a piecewise constant function $w_\ve$ such that 
\[
\int_{-P}^{P}w_{\ve}(x)dx~=~0,\qquad\quad \lim_{\ve\to 0+}\big\|w_{\ve}\big\|_{{\bf L}^1([-P,P])}~=~ \|w\|_{{\bf L}^1([-P,P])},
\]
and 
\[
\TV\{w_{\ve};[-P,P]\}~\leq~\TV\{w;[-P,P]\}.
\]
Since the average of $w_{\ve}$ is zero in $[-P,P]$, by slicing the graph of $w_\ve$ horizontally, we can write $w_\ve$ (restricted to
one period $[-P,P]$) as a sum of characteristic functions
\[
w_{\ve}(x)~=~\sum_{i=1}^N\rho_i\chi_{[a_i,b_i]}(x),\qquad x\in [-P,P],
\]
such that 
\[
\big\|w_{\ve}\big\|_{{\bf L}^1([-P,P])}~=~\sum_{i=1}^{N}|\rho_i| (b_i-a_i),\qquad \TV\{w_{\ve};[-P,P]\}~\leq~2\sum_{i=1}^N|\rho_i|.
\]
By (\ref{L-c}), we have
\bel{esGper1}
\begin{split}
\big\|{\bf G}_{\mathrm{per}}[w_{\ve}]&\big\|_{{\bf L}^1([-P,P])}~\leq~\sum_{i=1}^{N}|\rho_i|\left\|{\bf G}_{\mathrm{per}}[\chi_{[a_i,b_i]}]\right\|_{{\bf L}^1([-P,P])} \\[2mm]
&~\leq~C(2+5\ln 2+2\ln P) \sum_{i=1}^{N} |\rho_i| (b_i-a_i)- 2C\sum_{i=1}^N |\rho_i| (b_i-a_i)\ln (b_i-a_i)\\
&~\leq~C\big(2+5\ln 2+2\ln P\big)\big\|w_{\ve}\big\|_{{\bf L}^1([-P,P])}- 2C\sum_{i=1}^N |\rho_i| (b_i-a_i)\ln (b_i-a_i).
\end{split}
\eeq
Set $\rho=\ds\sum_{i=1}^N|\rho_i|$. By applying Jensen’s inequality to the convex function $x\mapsto x\ln x$, we get 
\[
\begin{split}
\sum_{i=1}^N|\rho_i| (b_i-a_i)&\ln (b_i-a_i)~=~\rho\cdot\sum_{i=1}^N{|\rho_i|\over \rho} (b_i-a_i)\ln (b_i-a_i)\\
&~\geq~\rho\cdot \left(\sum_{i=1}^{N}|\rho_i|(b_i-a_i)/\rho\right)\ln\left(\sum_{i=1}^{N}|\rho_i|(b_i-a_i)/\rho\right)\\
&~\geq~\big\|w_{\ve}\big\|_{{\bf L}^1([-P,P])}\cdot\left(\ln\big\|w_{\ve}\big\|_{{\bf L}^1([-P,P])}-\ln\big( \TV\{w_{\ve};[-P,P]\}\big)+\ln 2 \right),
\end{split}
\]
and (\ref{esGper1}) yields 
\begin{multline*}
\big\|{\bf G}_{\mathrm{per}}[w_{\ve}]\big\|_{{\bf L}^1([-P,P])}\\~\leq~C\big\|w_{\ve}\big\|_{{\bf L}^1([-P,P])}\left(2+3\ln 2+2\ln P+2\ln\big( \TV\{w_{\ve};[-P,P]\}\big)- 2\ln\big\|w_{\ve}\big\|_{{\bf L}^1([-P,P])}\right).
\end{multline*}
Finally, taking $\ve$ to $0+$, we achieve (\ref{L-e-p}).
\endproof
\medskip

From Lemma \ref{BV-L1}, one can follow the same argument in \cite[Theorem 4.3]{BN} to obtain a partial result on uniqueness. 

\begin{theorem}\label{p-uniqueness} Assume that conditions {\bf (F)} and {(\bf K)} hold. Let $u,v$ be  entropy weak solutions of the spatially periodic Cauchy problem (\ref{EE}), with the same initial data $u(0,\cdot)=v(0,\cdot)=\bar{u}\in {\bf L}^2_{\mathrm{per}}(\R)$. If  the total variation of both $u(t,\cdot)$ and $v(t,\cdot)$ over $[-P,P]$ remains uniformly bounded in $[0,T]$, then $u$ and $v$ coincide on $[0,T]\times \R$.
\end{theorem}
{\bf Proof.} For every $w\in {\bf L}^2_{\mathrm{per}}(\R)\cap \mathcal{C}^1(\R)$, it holds 
\[
\begin{split}
\int_{-P}^{P}{\bf G}_{\mathrm{per}}[w](x)dx&=~\int_{-P}^{P}\left(\lim_{\ve\to 0+}\int_{(-P,P)\backslash (-\ve,\ve)}K(y) \, w(x-y)dy\right)dx\\
&~=~\int_{-P}^{P}\left(\lim_{\ve\to 0+}\int_{\ve}^{P} K(y)\, \big[w(x-y)-w(x+y)\big]dy\right)dx.
\end{split}
\]
By (\ref{ap1}), we have  
\[
\big|K(y)\cdot\big[w(x-y)-w(x+y)\big]\big|~\leq~|yK(y)|\cdot \left|{w(x-y)-w(x+y)\over y}\right|~\leq~2C_K\|w\|_{\mathcal{C}^1(\R)},
\]
which implies that 
\bel{11}
\int_{-P}^{P}{\bf G}_{\mathrm{per}}[w](x)dx~=~\lim_{\ve\to 0+}\int_{\ve}^{P}K(y)\left(\int_{-P}^{P} \big[w(x-y)-w(x+y)\big]dx\right)dy~=~0.
\eeq
By the density of $ {\bf L}^2_{\mathrm{per}}(\R)\cap \mathcal{C}^1(\R)$ in ${\bf L}^2_{\mathrm{per}}(\R)$ and the continuity of ${\bf G}_{\mathrm{per}}$, we get
\[
\int_{-P}^{P}{\bf G}_{\mathrm{per}}[w](x)dx=0 \qquad\forall w\in {\bf L}^2_{\mathrm{per}}(\R).
\]
Hence, since $u,v$ are both weak solutions of (\ref{EE}) with $u(t,\cdot),v(t,\cdot)\in {\bf L}^2_{\mathrm{per}}$, their average value remains constant in time
\[
\int_{-P}^{P}u(t,x)~=~\int_{-P}^{P}u(0,x)dx~=~\int_{-P}^{P}\bar{u}dx~=~\int_{-P}^{P}v(0,x)dx~=~\int_{-P}^{P}v(t,x).
\]
In particular, the function  $w\doteq u-v$ is in ${\bf L}^2_{\mathrm{per}}(\R)$ and satisfies 
\[
  \sup_{t\in [0,T]}\tv\{w(t,\cdot);[-P,P]\}~=~\beta~<~\infty,\qquad \int_{-P}^{P}w(t,x)dx~=~0\qquad\forall t\in [0,T].
\]
By Lemma \ref{BV-L1}, setting $Z(t)\doteq \|u(t,\cdot)-v(t,\cdot)\|_{{\bf L}^1([-P,P])}$, we get
\bel{esG1}
\big\|{\bf G}[w(t,\cdot)]\big\|_{{\bf L}^1([-P,P])}~\leq~CZ(t)\, \left[2+3\ln 2+2\ln P+2\ln\beta-2\ln Z(t)\right].
\eeq
The uniform BV bounds on $u,v$ imply that the maps $t\mapsto u(t,\cdot)$ and  $t\mapsto v(t,\cdot)$ are both Lipschitz continuous with values in ${\bf L}^1([-P,P])$. Thus, the map $t\mapsto Z(t)$ is also Lipschitz continuous and a.e. differentiable in $[0,T]$. Finally, since the conservation law $u_t+f(u)_x=0$ generates a contractive semigroup, we obtain 
\[
\begin{split}
{d\over dt}Z(t)&~\leq~\big\|{\bf G}[w(t,\cdot)]\big\|_{{\bf L}^1([-P,P])}~\leq~CZ(t)\, \left[2+3\ln 2+2\ln P+2\ln\beta-2\ln Z(t)\right]. 
\end{split}
\]
By  Osgood's criterion, $Z(0)=0$ implies $Z(t)=0$ for all $t\in [0,T]$. This establishes the uniqueness of BV solutions in the spatially periodic case.
\endproof

\subsection{${\bf L}^1$-kernel $K$}\label{Uni-L1}
In this subsection, we study the well-posedness and generalized BV regularity entropy weak solutions to  the nonlinear balance laws with a singular source term (\ref{BL}) for kernel $K=K_2\in {\bf L}^1(\R)$ satisfying 
\bel{K1-1}
\|K\|_{{\bf L}^1(\R)}~\leq~L_K,\qquad \left|K^{(i)}(x)\right|~\leq~{C_K\over |x|^{i+1}}\qquad \text{for} \, i=0,1.
\eeq
The flux $f$ satisfies assumption {\bf (F)} for some exponents  $1\leq p_1\leq p_2 < p_1+1$ such that 
\bel{t-gamm}
\tilde{\gamma}_p~\doteq~{p_2\over p_1+1}~\in~(0,1).
\eeq
In this framework, an entropy weak solution of the Cauchy problem (\ref{BL}) with initial data 
\bel{ini-da}
u(0,\cdot)~=~\bar{u}~\in~{\bf L}^1(\R)
\eeq
is defined as follows:
\begin{definition}\label{Def2} A function $u\in \L^1_{loc}([0,\infty[\,\times\R)\cap {\bf L}^\infty_{loc}\big(]0,\infty[, {\bf L}^{\infty}(\R)\big)$ is an entropy weak solution of (\ref{BL})-(\ref{G}) with initial condition (\ref{ini-da}) if $u$ satisfies the following properties: 
\begi
\item[(i)] The map $t\mapsto u(t,\cdot)$ is continuous with values in $\L^1(\R)$ and
satisfies the initial condition in (\ref{ini-da}).
\item[(ii)]  For any $k\in \R$ and every nonnegative 
test function $\phi\in\C^1_c(]0,\infty[\,\times\R)$
one has
\bel{ei1}\dint  \Big[ |u-k|\phi_t + \Big(f(u)-f(k)\Big)\hbox{\rm sign}(u-k) \phi_x
+{\bf G}[u(t,\cdot)](x)\hbox{\rm sign}(u-k)\phi  \Big]\, dxdt~\geq~0.\eeq
\endi
\end{definition}

To prove this, we first assume  $\bar{u}\in {\bf L}^1(\R)\cap  {\bf L}^2(\R)$ and recall the approximating solution   $u^{\nu}$ defined in  (\ref{u-nu}):
\[
\begin{cases}
u^{\nu}(0,\cdot)~=~\bar{u},\qquad u^{\nu}(t_{\ell},\cdot)~=~u^{\nu}(t_\ell-,\cdot)+2^{-\nu}\cdot {\bf G}[u^{\nu}(t_\ell-,\cdot)],\qquad \ell=1,2,\dots,\\[4mm]
u^{\nu}(t,\cdot)~=~S_{t-t_\ell}\big(u^{\nu}(t_\ell,\cdot)\big),\qquad t\in [t_\ell,t_{\ell+1}[,\quad \ell=0,1,2,\dots
\end{cases}
\]
For every $\ell\geq 1$, one has
\[
\begin{split}
\big\|u^{\nu}(t_\ell,\cdot)\big\|_{{\bf L}^1(\R)}& ~\leq~ \bigl(1+2^{-\nu} \, \|K\|_{{\bf L}^1(\R)}\bigr) \cdot \|u^{\nu}(t_\ell-,\cdot)\|_{{\bf L}^1(\R)}\\[2mm]
& ~\leq~ \bigl(1+2^{-\nu}L_K\bigr) \cdot \|u^{\nu}(t_{\ell-1},\cdot)\|_{{\bf L}^1(\R)} ~\leq~ \bigl(1+2^{-\nu} L_K\bigr)^{2^{\nu}t_\ell} \cdot \|\bar{u}\|_{{\bf L}^1(\R)}
\end{split}
\]
and this yields 
\bel{u-L1}
\big\|u^{\nu}(t,\cdot)\big\|_{{\bf L}^1(\R)}~\leq~ \bigl(1+2^{-\nu}L_K\bigr)^{2^{\nu}t} \cdot \|\bar{u}\|_{{\bf L}^1(\R)} ~\leq~ e^{L_K t} \cdot \|\bar{u}\|_{{\bf L}^1(\R)} \quad\forall t\geq 0.
\eeq
Using (\ref{u-L1}) and Lemma (\ref{F-BV}), we establish  a H\"older continuity on  the backward characteristics which are defined in the previous section.
\begin{lemma}\label{bb-c} For every given $\nu>0$,  let $t\mapsto x(t)$ be any characteristic on $[0,1]$ for the approximate solution $u^{\nu}$. For every $0\leq \tau_1<\tau_2\leq 1$, it holds 
\bel{Holder-11}
|x(\tau_2)-x(\tau_1)|~\leq~\Gamma_3\cdot (\tau_2-\tau_1)^{1-\tilde{\gamma}_p},
\eeq
where the constant $\Gamma_3> 0$ is computed by 
\bel{gamma-T-11}
\Gamma_3~\doteq~C_f\cdot \left[2^{p_1}\cdot \left(1+{1+p_1\over C_f}\cdot \left(f(0)+2 e^{L_K} \cdot \|\bar{u}\|_{{\bf L}^1(\R)} \right)\right)\right]^{\tilde{\gamma}_p}.
\eeq
\end{lemma}
{\bf Proof.} The proof follows the same argument as in Lemma \ref{hc-bb-u} and we shall give a brief proof. Recalling the definitions of $u^{\nu,\pm}$ and $x^{\pm}$ from that proof, we have
\[
x^{-}(\tau_2)-x^{-}(\tau_1)~\leq~x(\tau_2)-x(\tau_1)~\leq~x^{+}(\tau_2)-x^{+}(\tau_1).
\]
Assume that  $x^{-}(\tau_2)-x^{-}(\tau_1)\geq 0$. Then we have 
\[
|x(\tau_2)-x(\tau_1)|~\leq~x^{+}(\tau_2)-x^{+}(\tau_1).
\]
In this case, we consider the integral of  $u^{\nu,+}$ over the domain $\{(s,z):s\in [\tau_1,\tau_2],z\geq x^+(s)\}$. Applying the divergence theorem to  $\big(u,f(u)\big)$, we obtain that 
\bel{td}
\begin{split}
0&~\leq~\int^{\infty}_{x^+(\tau_2)}u^{\nu,+}(\tau_2,z)~dz~\leq~\int^{\infty}_{x^+(\tau_1)}u^{\nu,+}(\tau_2,z)dz+\sum_{\tau_1<t_\ell\leq \tau_2}\int^{\infty}_{x^+(t_{\ell})}\Big(u^{\nu,+}(t_\ell,z)-u^{\nu,+}(t_\ell-,z)\Big)dz\\
&\qquad\qquad\qquad \qquad\qquad\qquad\qquad\qquad+\int_{\tau_1}^{\tau_2}\bigg(f\Big(u^{\nu,+}\big(s,x^+(s)\big)\Big)-u^{\nu,+}\big(s,x^+(s)\big)\cdot \dot{x}^+(s)\bigg)ds\\
&~\leq~ (1+\tau_2) e^{L_K\tau_2} \cdot \|\bar{u}\|_{{\bf L}^1(\R)} + \int_{\tau_1}^{\tau_2}\bigg(f\Big(u^{\nu,+}\big(s,x^+(s))\big)-u^{\nu,+}\big(s,x^+(s)\big)\cdot f'\Big(u^{\nu,+}\big(s,x^+(s)\big)\Big)\bigg)ds,
\end{split}
\eeq
For every $\omega\geq 0$, it holds 
\[
\begin{split}
\omega f'(\omega)-f(\omega)&~=~-f(0)+\omega f'(\omega)-\int_{0}^{1}f'(s\omega)\omega~ds~\geq~-f(0)+C_f\omega^{1+p_1}\cdot \int_{0}^{1}(1-s)^{p_1}ds\\
&~=~-f(0)+{C_f\over 1+p_1}\cdot \omega^{1+p_1}~\geq~
-f(0)+{C_f\over (1+p_1)}\cdot \left[{{|f'(\omega)|^{1/\tilde{\gamma}_p}}\over 2^{p_1}C_f^{1/\tilde{\gamma}_p}}-1\right].
\end{split}
\]
Hence, from (\ref{td}) we derive that 
\[
\int_{\tau_1}^{\tau_2}\bigg({|\dot{x}^{+}(s)|^{1/\tilde{\gamma}_p}\over 2^{p_1}C_f^{1/\tilde{\gamma}_p}}-1~\bigg)ds~\leq~{1+p_1\over C_f}\cdot \left(f(0)+2 e^{L_K} \cdot \|\bar{u}\|_{{\bf L}^1(\R)} \right),
\]
which yields (\ref{Holder-11}).  

The case $x^{+}(\tau_2)-x^{+}(\tau_1)\leq 0$ can be treated analogously, completing the proof.
\endproof
\medskip

%
Toward establishing the well-posedness and BV regularity of the weak entropy solution of (\ref{BL}), we shall derive an Oleinik-type inequality for the sequence of approximating solutions $(u^{\nu})_{\nu\geq 1}$  defined in  (\ref{u-nu}).
%
%
\begin{lemma} There exists a constant $\Gamma_4>0$ such that for every $\nu\in \mathbb{Z}^+$ and $\tau=\ell_02^{-\nu}$ for $\ell_0\in \{1,\cdots, 2^{\nu}\}$, it holds 
\bel{O-le1}
f'\big(u^{\nu}(\tau-,x_2)\big)-f'\big(u^{\nu}(\tau-,x_1)\big)~\leq~\max\left\{{4(x_2-x_1)\over \tau},  2\, \sqrt{M_\tau(x_2-x_1)}\right\},\qquad x_1<x_2,
\eeq
with 
\bel{M-r}
M_\tau~=C_{\tau/2}\|f''\|_{{\bf L}^{\infty}(-C_{{r/ 2}},C_{r/ 2})},\qquad C_s~=~\big(1+ \|K\|_{{\bf L}^1(\R)}\big)\cdot \left(\Gamma_4+{\Gamma_4\over s^{p_1+1}}\right).
\eeq
\end{lemma}
{\bf Proof.}  {\bf 1.} From Remark \ref{L1-1} and Lemma \ref{bb-c}, and using the same argument as in the proof of Lemma \ref{L-infty-b}, one can show that there exists a constant $\Gamma_4>0$ such that for every $\nu\in \mathbb{Z}^+$ and $s=\ell 2^{-\nu}$ for $\ell\in \{1,\cdots, 2^{\nu}\}$, it holds 
\bel{L-infty-1}
\left\|u^{\nu}(s-,\cdot)\right\|_{{\bf L}^{\infty}(\R)}~\leq~ \Gamma_4\cdot \left(1+{1\over s^{p_1+1}}\right).
\eeq
Hence, by (\ref{K1-1}), we get
\bel{L-infty-2}
\begin{split}
\left\|u^{\nu}(s+,\cdot)\right\|_{{\bf L}^{\infty}(\R)}&~\leq~\left\|u^{\nu}(s-,\cdot)\right\|_{{\bf L}^{\infty}(\R)}+2^{-\nu}\cdot  \left\|{\bf G}[u^{\nu}(s-,\cdot)]\right\|_{{\bf L}^\infty(\R)}\\[2mm]
&~\leq~\left(1+2^{-\nu}\cdot \|K\|_{{\bf L}^1(\R)}\right)\cdot \left\|u^{\nu}(s-,\cdot)\right\|_{{\bf L}^{\infty}(\R)}\\[2mm]
&~\leq~\Gamma_4\big(1+ \|K\|_{{\bf L}^1(\R)}\big)\cdot \left(1+{1\over s^{p_1+1}}\right)~=~C_s.
\end{split}
\eeq
{\bf 2.} Fix $\tau=\ell_02^{-\nu}$ for some $\ell_0\in \{1,\cdots, 2^{\nu}\}$. For every $x_1<x_2$ and $\nu\geq 1$, let $x_i^{\nu}$ be the minimal backward characteristic for the approximate solution $u^{\nu}$, through the point $(\tau,x_i)$ for $i=1,2$. For every $t \in [\tau/2,\tau]$, setting $k_t=\left\lfloor 2^{\nu} t\right\rfloor$ the integer part of $ 2^{\nu} t$, we get 
\[
\begin{split}
x_i&~=~x^{\nu}_i(t)+(t_{k_t+1}-t) f'\Big(u^{\nu}\big(t_{k_t+1}-,x^{\nu}_i(t_{k_t+1})\big)\Big)+\sum_{k_t+1 \leq \ell \leq \ell_0} 2^{-\nu}f'\Big(u^{\nu}\big(t_\ell+,x^{\nu}_i(t_\ell)\big)\Big),
\end{split}
\]
which implies
\[
\begin{split}
x_2-x_1&~=~x_2^{\nu}(t)-x_1^{\nu}(t)+(t_{k_t+1}-t) \Big[f'\Big(u^{\nu}\big(t_{k_t+1}-,x^{\nu}_2(t_{k_t+1})\big)\Big)-f'\Big(u^{\nu}\big(t_{k_t+1}-,x^{\nu}_1(t_{k_t+1})\big)\Big)\Big]\\[2mm]
&\qquad\qquad\qquad\quad~~+\sum_{k_t+1 \leq \ell \leq \ell_0}2^{-\nu}\Big[f'\Big(u^{\nu}\big(t_\ell+,x^{\nu}_2(t_\ell)\big)\Big)-f'\Big(u^{\nu}\big(t_\ell+,x^{\nu}_1(t_\ell)\big)\Big)\Big].
\end{split}
\]
By (\ref{L-infty-1}), (\ref{L-infty-2}) and (\ref{M-r}), for all $ k_t+1\leq\ell\leq\ell_0$ we get 
\[
\begin{split}
\Big|f'\big(u^{\nu}(\tau,x_i)\big)&-f'\Big(u^{\nu}\big(t_\ell+,x^{\nu}_i(t_\ell)\big)\Big)\Big|~=~\left|\sum_{j=\ell+1}^{\ell_0}\Big[f'\Big(u^{\nu}\big(t_j+,x^{\nu}_i(t_j)\big)\Big)-f'\Big(u^{\nu}\big(t_j-,x^{\nu}_i(t_j)\big)\Big)\Big]\right|\\
&\leq~\|f''\|_{{\bf L}^{\infty}(-C_{{r/ 2}},C_{r/ 2})}\cdot \sum_{j=\ell+1}^{\ell_0} \big|u^{\nu}\big(t_j+,x^{\nu}_i(t_j)\big)-u^{\nu}\big(t_j-,x^{\nu}_i(t_j)\big)\big|\\
&\leq~\|f''\|_{{\bf L}^{\infty}(-C_{{r/ 2}},C_{r/ 2})}\cdot \sum_{j=\ell+1}^{\ell_0}2^{-\nu} \left\|u^{\nu}(t_j-,\cdot)\right\|_{{\bf L}^{\infty}(\R)}\\
&\leq~M_{\tau}\cdot  \sum_{j=\ell+1}^{\ell_0}2^{-\nu}~\leq~M_{\tau}\cdot(\tau-t_{\ell}).
\end{split}
\]
Similarly,
\[
\Big|f'\big(u^{\nu}(t,x_i(t))\big)-f'\Big(u^{\nu}\big(t_\ell+,x^{\nu}_i(t_\ell)\big)\Big)\Big|~\leq~ M_\tau(t_\ell-t).
\]
Hence, for all $t\in [\tau/2,\tau]$, it holds
\bel{ek1}
\begin{cases}
x_2-x_1~\geq~x_2^{\nu}(t)-x_1^{\nu}(t)+(\tau-t)\Big[f'\big(u^{\nu}(\tau,x_2)\big)-f'\big(u^{\nu}(\tau,x_1)\big)\Big]-M_\tau\cdot (\tau-t)^2,\\[2mm]
x_2-x_1~\leq~x_2^{\nu}(t)-x_1^{\nu}(t)+(\tau-t)\Big[f'\big(u^{\nu}(\tau,x_2)\big)-f'\big(u^{\nu}(\tau,x_1)\big)\Big]+M_\tau\cdot (\tau-t)^2,\\[2mm]
x_2-x_1~\leq~x_2^{\nu}(t)-x_1^{\nu}(t)+(\tau-t)\Big[f'\big(u^{\nu}(t,x_2)\big)-f'\big(u^{\nu}(t,x_1)\big)\Big]+M_\tau\cdot (\tau-t)^2.
\end{cases}
\eeq
In particular, we have 
\[
f'\big(u^{\nu}(\tau,x_2)\big)-f'\big(u^{\nu}(\tau,x_1)\big)~\leq~{x_2-x_1\over \tau-t}+ M_\tau\cdot (\tau-t),\qquad t\in [\tau/2,\tau).
\]
If $x_2-x_1\geq \ds {M_\tau \tau^2\over 4}$, then by choosing $t=\tau/2$, we get
\[
f'\big(u^{\nu}(\tau,x_2)\big)-f'\big(u^{\nu}(\tau,x_1)\big)~\leq ~ {4(x_2-x_1)\over \tau} .
\]
Otherwise, if $x_2-x_1< \ds {M_\tau \tau^2\over 4}$, then by choosing $t\in [\tau/2,\tau]$ such that $\tau-t= \ds\sqrt{{x_2-x_1\over M_\tau}}$, we get 
\[
f'\big(u^{\nu}(\tau,x_2)\big)-f'\big(u^{\nu}(\tau,x_1)\big)~\leq~\ds 2\, \sqrt{M_{\tau}(x_2-x_1)},
\]
which yields 
\[
f'\big(u^{\nu}(\tau,x_2)\big)-f'\big(u^{\nu}(\tau,x_1)\big)~\leq~\max\left\{{4(x_2-x_1)\over \tau}, 2\, \sqrt{M_{\tau}(x_2-x_1)}\right\}.
\]
Taking $\nu\to\infty$, we obtain (\ref{O-le1}).
\endproof

\begin{theorem}\label{Un-L1} Assume that (\ref{K1-1})-(\ref{t-gamm}) hold. For every $\bar{u}\in {\bf L}^1(\R)$, the Cauchy problem (\ref{BL})-(\ref{G}) with initial condition (\ref{ini-da}) admits a unique entropy weak solution $u(\tau,x)$ defined for all $(\tau,x)\in ]0,\infty[\times\R$. Moreover, for every $\tau\in [0,T]$, there exists a constant $C_T>0$ such that
\bel{dcd1}
\|u(\tau,\cdot)\|_{{\bf L}^1(\R)}~\leq~ e^{L_K \tau} \cdot \|\bar{u}\|_{{\bf L}^1(\R)},\qquad \|u(\tau,\cdot)\|_{{\bf L}^\infty(\R)}~\leq~C_T\cdot \left(1+{1\over \tau^{1/(p_1+1)}}\right),
\eeq
and for every $x_1<x_2$
\bel{dcd2}
f'\big(u(\tau,x_2)\big)-f'\big(u(\tau,x_1)\big)~\leq~\max\left\{{4(x_2-x_1)\over \tau},  2\, \sqrt{C_T(x_2-x_1)}\right\}.
\eeq
In addition,  $u(\tau,\cdot)$  belongs to $BV_{loc}^{{1\over 2p_1}}(\R)$  for all $\tau>0$.
%
\end{theorem}

{\bf Proof.} 
{\bf 1.} For every $\bar{u}\in {\bf L}^1(\R)\cap {\bf L}^2(\R)$, let $(u^{\nu})_{\nu\geq 1}$ be the sequence of approximating solutions  defined in   (\ref{u-nu}). By the proof of Theorem \ref{Main1}, there exists a countably infinite set of indices $\mathcal{I}\subseteq \mathbb{N}$ such that $\big(u^{\nu}\big)_{\nu\in\mathcal{I}}$ converges to  a weak entropy solution $u$ of (\ref{BL}) in the sense of Definition  \ref{Def1}, with initial condition $u(0,\cdot)=\bar{u}$. Moreover,  both (\ref{l1}) and (\ref{l2}) hold. Consequently, by (\ref{u-L1}) and (\ref{L-infty-1}), for every $T>0$ there exists a constant $C_T>0$ such that  both (\ref{dcd1}) and (\ref{dcd2}) are satisfied for all $\tau\in (0,T]$. Hence, since the map $\tau\mapsto u(\tau,\cdot)$ is  continuous from $(0,\infty)$ to  $\L^2(\R)$, it is also continuous from  $(0,\infty)$ to    $\L^1(\R)$. On the other hand, the continuity also holds at $\tau=0$, that means
\bel{L1-0}
\lim_{\tau\to 0+} \|u(\tau,\cdot)-\bar{u}\|_{{\bf L}^1(\R)}~=~0.
\eeq
Indeed, for any $\tau>0$, $\nu\in \mathbb{N}$ and $\ell\in \{1,2,\cdots, \lfloor 2^{\nu}\tau\rfloor\}$, we have 
\[
\begin{split}
\big\|S_{\tau-t_\ell}\big(u^{\nu}(t_\ell,\cdot)\big)&-S_{\tau-t_{\ell-1}}\big(u^{\nu}(t_{\ell-1},\cdot)\big)\big\|_{{\bf L}^1(\R)}~=~\big\|S_{\tau-t_\ell}\big(u^{\nu}(t_\ell,\cdot)\big)-S_{\tau-t_{\ell}}\big(u^{\nu}(t_{\ell}-,\cdot)\big)\big\|_{{\bf L}^1(\R)}\\[1mm]
&\leq~\big\|u^{\nu}(t_\ell,\cdot)-u^{\nu}(t_\ell-,\cdot)\big\|_{{\bf L}^1(\R)}~\leq~2^{-\nu}\|K\|_{{\bf L}^1(\R)}\cdot \|u^{\nu}(t_\ell-,\cdot)\|_{{\bf L}^1(\R)}\\[1mm]
&\leq~2^{-\nu}L_Ke^{L_K\tau}\cdot \|\bar{u}\|_{{\bf L}^1(\R)},
\end{split}
\]
which yields
\[
\begin{split}
\big\|u^{\nu}(\tau,\cdot)-S_\tau(\bar{u})\big\|_{{\bf L}^1(\R)}&~\leq~\sum_{i=1}^{\lfloor 2^{\nu}\tau\rfloor}\big\|S_{\tau-t_\ell}\big(u^{\nu}(t_\ell,\cdot)\big)-S_{\tau-t_{\ell-1}}\big(u^{\nu}(t_{\ell-1},\cdot)\big)\big\|_{{\bf L}^1(\R)}\\
&~\leq~L_Ke^{L_K\tau} \|\bar{u}\|_{{\bf L}^1(\R)}\cdot\tau.
\end{split}
\]
Taking $\mathcal{I}\ni\nu \to \infty$, we get
\[
\big\|u(t,\cdot)-S_t(\bar{u})\big\|_{{\bf L}^1(\R)}~\leq~L_Ke^{L_K\tau} \|\bar{u}\|_{{\bf L}^1(\R)}\cdot\tau.
\]
Hence, (\ref{L1-0})  follows from the continuity of the semigroup $S_\tau$ at time $\tau=0$ and $u$ is also an entropy weak solution of  the Cauchy problem (\ref{BL}), (\ref{ini-da}) in the sense of Definition \ref{Def2}.
\medskip

{\bf 2.}  Next, let $u_1$ and $u_2$ be entropy weak solutions of (\ref{BL}) with $u_1(0,\cdot)=\bar{u}_1\in {\bf L}^1(\R)$ and $u_2(0,\cdot)=\bar{u}_2\in {\bf L}^1(\R)$ in the sense of Definition \ref{Def2}.  We  prove that
\bel{Sta}
\|u_2(\tau,\cdot)-u_1(\tau,\cdot)\|_{{\bf L}^1(\R)}~\leq~e^{L_K\tau}\cdot \|\bar{u}_2-\bar{u}_1\|_{{\bf L}^1(\R)}.
\eeq
Since the map $\tau\mapsto u_i(\tau,\cdot)$ is continuous with values in ${\bf L}^1(\R)$, for every $\ve>0$ there exists $\tau_\ve>0$ such that 
\[
\|u_2(\tau_{\ve},\cdot)-u_1(\tau_{\ve},\cdot)\|_{{\bf L}^1(\R)}~\leq~\ve+\|\bar{u}_2-\bar{u}_1\|_{{\bf L}^1(\R)}.
\]
Notice that by Definition \ref{Def2},  it holds
\[
\sup_{\tau\in [\tau_{\ve},T]} \|u_i(\tau,\cdot)\|_{{\bf L}^2(\R)}~<~\infty\qquad\forall T>\tau_{\ve}.
\]
Set $Z(\tau)\doteq \|u_2(\tau,\cdot)-u_1(\tau,\cdot)\|_{{\bf L}^1(\R)}$. With the same argument as in the proof of \cite[Theorem 6.2]{Bbook}, for every $\tau_\ve\leq s<\tau$, we have
\[
Z(\tau)~\leq~Z(s)+\int_{s}^{\tau}\|{\bf G}[u_2(t,\cdot)-u_1(t,\cdot)]\|_{{\bf L}^1(\R)}~dt~\leq~Z(s)+L_K\cdot\int_{s}^{\tau} Z(t)~dt.
\]
Applying Gronwall's inequality, we get
\[
Z(\tau)~\leq~e^{L_K(\tau-\tau_\ve)} Z(\tau_\ve)~\leq~e^{L_K(\tau-\tau_\ve)} \cdot \big(\ve+\|\bar{u}_2-\bar{u}_1\|_{{\bf L}^1(\R)}\big) \qquad\forall \tau\geq \tau_{\ve},
\]
which yields (\ref{Sta}). Combining this with the existence of solutions for initial data in ${\bf L}^1(\R)\cap {\bf L}^2(\R)$, we conclude that for every $\bar{u}\in {\bf L}^1(\R)$,  the Cauchy problem (\ref{BL}), (\ref{ini-da}) admits a unique entropy weak solution $u$ in the sense of Definition \ref{Def2}.

\medskip

{\bf 3.} It remains to show that $u(\tau,\cdot)$  belongs to $BV_{loc}^{1/(2 p
_1)}(\R)$ for every $\tau>0$. Fix $\nu\in \mathcal{I}$ and $\tau=\ell_02^{-\nu}$ for $\ell_0\in \{1,\cdots, 2^{\nu}\}$. Given any  $a<b$, let $\{x_0,x_1,\cdots, x_N\}$ be a partition of $[a,b]$ such that $u$ is continuous at $x_i$ for all $i\in\{1,\cdots, N\}$. For every $i\in\{1,\cdots, N\}$, let $x_i^{\nu}$ be the backward characteristic through $(\tau,x_i)$. Set $\beta_i=u^{\nu}(\tau,x_{i})-u^{\nu}(\tau,x_{i-1})$. We shall denote by
\[
\begin{cases}
I^+~=~\{i\in \{1,\dots, N\}: \beta_i\geq 0\},\qquad I^-~=~\{i\in \{1,\dots N\}: \beta_i<0 \},\\[2mm]
I^+_\tau~=~\{i\in I^+:x_i-x_{i-1}\geq M_\tau \tau^2/4\},\qquad I^-_\tau~=~\{i\in I^-:\Phi_f(|\beta_i|)\geq M_\tau  \tau\},
\end{cases}
\]
with $\Phi_f(\cdot)$ defined in (\ref{Phi}) as
\bel{PPPP}
\Phi_f(s)~\doteq~\inf_{a \in\R} \big\{f'(a+s)-f'(a)\big\}~\geq~C_f s^{p_1}~\quad\forall s>0.
\eeq
From (\ref{O}), (\ref{Phi}), and (\ref{O-le1}), the following hold:
\begin{itemize}
\item If $i\in I^+_\tau$, then 
\[
\Phi^2_f(\beta_i)~\leq~\big[f'\big(u^{\nu}(\tau,x_i)\big)-f'\big(u^{\nu}(\tau,x_{i-1})\big)\big]^2~\leq~{16 (x_i-x_{i-1})^2\over \tau^2}.
\]
\item Otherwise, if $i\in I^{+}\backslash I^+_\tau$, then 
\[
\Phi^2_f(\beta_i)~\leq~\big[f'\big(u^{\nu}(\tau,x_i)\big)-f'\big(u^{\nu}(\tau,x_{i-1})\big)\big]^2~\leq~4M_\tau (x_i-x_{i-1}).
\]
\end{itemize}
Hence, we estimate 
\bel{eet1}
\begin{split}
\sum_{i\in I^+}\Phi^2_f(\beta_i)&~=~\sum_{i\in I^+_\tau}\Phi^2_f(\beta_i)+\sum_{i\in I^+\backslash I^+_\tau}\Phi^2_f(\beta_i)\\
&~\leq~\sum_{i\in I^{+}_\tau}{16 (x_i-x_{i-1})^2\over \tau^2}+\sum_{i\in I^+\backslash I^{+}_\tau} 4 M_{\tau} (x_i-x_{i-1})\\
&~\leq~ {16 (b-a)^2\over \tau^2}+4M_\tau(b-a).
\end{split}
\eeq

Next, from the second inequality in (\ref{ek1}), for all $i\in I^{-}, t\in [\tau/2,\tau]$ it holds
\bel{e1k1}
\Phi_f(|\beta_i|)~\leq~-\big[f'\big(u^{\nu}(\tau,x_{i})\big)-f'\big(u^{\nu}(\tau,x_{i-1})\big)\big]~\leq~{x^{\nu}_i(t)-x^{\nu}_{i-1}(t)\over \tau-t}+ M_\tau(\tau-t).
\eeq
Two cases are considered:
\begin{itemize}
\item If  $i\in I_\tau^{-}$, then  $\Phi_f(|\beta_i|)\geq M_\tau \tau$  and by  choosing $t=\tau/2$ in (\ref{e1k1}), we derive
\[
x^{\nu}_i(\tau/2)-x^{\nu}_{i-1}(\tau/2)~\geq~{\Phi_f(|\beta_i|)\tau\over 2}- {M_\tau \tau^2\over 4 }~\geq~{\Phi_f(|\beta_i|)\tau\over 4}.
\]
\item Otherwise, if $i \in I^- \setminus I_\tau^{-}$, then $\Phi_f(|\beta_i|) < M_\tau \tau$. Hence, there exists $t_i \in [\tau/2,\tau]$ such that
\[
\tau - t_i = \frac{\Phi_f(|\beta_i|)}{2M_\tau} < \frac{\tau}{2}.
\]
By \eqref{e1k1}, for all $t \in [t_i,\tau]$, we  have
\[
x^{\nu}_i(t) - x^{\nu}_{i-1}(t)~\ge~\Phi_f(|\beta_i|)(\tau - t) - M_\tau (\tau - t)^2~=~ \frac{\Phi_f(|\beta_i|)}{2}(\tau - t),
\]
which particularly implies that  
\bel{estq}
x^{\nu}_i(t_i) - x^{\nu}_{i-1}(t_i)~\geq~\frac{\Phi_f(|\beta_i|)}{2}(\tau - t_i)~=~\frac{\Phi^2_f(|\beta_i|)}{4M_{\tau}},
\eeq
and 
\[
x^{\nu}_i(t) - x^{\nu}_{i-1}(t)~\geq~M_{\tau} (\tau-t)^2.
\]
Recalling the third inequality in  (\ref{ek1}), we obtain 
\[
\begin{split}
\dot{x}^{\nu}_i(t) - \dot{x}^{\nu}_{i-1}(t)&~=~f'\big(u^{\nu}(t,x_i)\big)-f'\big(u^{\nu}(t,x_{i-1})\big)~\geq~{-1\over \tau-t}\cdot \left(x_{i}^{\nu}(t)-x_{i-1}^{\nu}(t)+M_\tau (\tau - t)^2\right)\\
&~\geq~ {-2 \left(x_{i}^{\nu}(t)-x_{i-1}^{\nu}(t)\right)\over \tau-t}~\geq~{-4 \left(x_{i}^{\nu}(t)-x_{i-1}^{\nu}(t)\right)\over \tau}\qquad\forall t\in [t_i,\tau],
\end{split}
\]
and  Gronwall's inequality yields 
\[
x^{\nu}_i(t_i) - x^{\nu}_{i-1}(t_i)~\leq~e^{4(\tau-t_i)\over \tau}\cdot \left(x^{\nu}_i(\tau) - x^{\nu}_{i-1}(\tau)\right)~\leq~e^2\cdot (x_{i}-x_{i-1}).
\]
From (\ref{estq}), it holds
\[
\Phi^2_f(|\beta_i|)~\leq~4M_{\tau}\left(x^{\nu}_i(t_i) - x^{\nu}_{i-1}(t_i)\right)~\leq~4e^2M_{\tau}\cdot (x_{i}-x_{i-1}).
\]
\end{itemize}
Thus, by Lemma \ref{bb-c} we estimate 
\bel{eqt1}
\begin{split}
\sum_{i\in I^{-}}\Phi^2_f(|\beta_i|)&~=~\sum_{i\in I^-_\tau}\Phi^2_f(|\beta_i|)+\sum_{i\in I^-\backslash I^-_\tau}\Phi^2_f(|\beta_i|)~\leq~\left(\sum_{i\in I^-_\tau}\Phi_f(|\beta_i|)\right)^2+\sum_{i\in I^-\backslash I^-_\tau}\Phi^2_f(|\beta_i|)\\[1mm]
&~\leq~\left(\sum_{i\in I_\tau^{-}}{4\big(x^{\nu}_i(\tau/2)-x^{\nu}_{i-1}(\tau/2)\big)\over \tau}\right)^2+4e^2M_\tau \sum_{I^-\backslash I_\tau^{-}} (x_i-x_{i-1}) \\[1mm]
&~\leq~{16\over \tau^2}\cdot \big[x^{\nu}_N(\tau/2)-x^{\nu}_{0}(\tau/2)\big]^2+4e^2M_\tau (b-a) \\[1mm]
&~\leq~{16\over \tau^2}\cdot\left[b-a+2\Gamma_3\right]^2+4e^2M_\tau (b-a).
\end{split}
\eeq
Combining with (\ref{eet1}), we get
\[
\sum_{i\in I^+\cup I^-}\Phi^2_f(|\beta_i|)~\leq~{16\over \tau^2}\, \left[(b-a)^2+(b-a+2\Gamma_3)^2\right]+4(e^2+1)M_{\tau}(b-a),
\]
and (\ref{PPPP}) yields  
\[
\sum_{i\in I^+\cup I^-}|\beta_i|^{2p_1}~\leq~{1\over C^2_f}\left({16\over \tau^2}\, \left[(b-a)^2+(b-a+2\Gamma_3)^2\right]+4(e^2+1)M_{\tau}(b-a)\right).
\]
Hence, $u^{\nu}(\tau,\cdot)$ belongs to $BV^{{1\over 2p_1}}_{loc}(\R)$ and 
\[
TV^{{1\over 2p_1}}\{u^{\nu}(\tau,\cdot);[a,b]\}~\leq~{1\over C^2_f}\left({16\over \tau^2}\, \left[(b-a)^2+(b-a+2\Gamma_3)^2\right]+4(e^2+1)M_{\tau}(b-a)\right).
\]
Taking the limit as $\nu\to\infty$, we obtain 
\bel{TV-b}
TV^{{1\over 2p_1}}\{u(\tau,\cdot);[a,b]\}~\leq~{1\over C^2_f}\left({16\over \tau^2}\, \left[(b-a)^2+(b-a+2\Gamma_3)^2\right]+4(e^2+1)M_{\tau}(b-a)\right),
\eeq
for every $\tau\in \{\ell 2^{-\mu}:\ell\in \{1,\cdots, 2^{\mu}\}, \mu\geq 1\}$ and $a<b$. Finally, by the continuity of the map $\tau\mapsto u(\tau,\cdot)$ in ${\bf L}^1(\R)$, this $TV^{1/(2p_1)}$-bound holds for all $\tau\in (0,1]$. Moreover, it extends to any interval $(0,T]$ with $T>0$, provided $u(\tau,\cdot)\in BV^{1/(2p_1)}_{\mathrm{loc}}(\R)$ for all $\tau>0$. This completes the proof.\endproof

\section{Local smoothness and wave breaking}\label{S4}
\setcounter{equation}{0}
In this final section, we study the  wave breaking of  the nonlocal balance law (\ref{BL}) with a singular source term (\ref{G}) for  initial data 
\bel{ini-H}
u(0,\cdot)~=~\bar{u}~\in~ H^2(\R),
\eeq 
under the following assumptions:
\begin{itemize}
\item [({\bf A1})] The  kernel $K$  is in $\mathcal{C}^1(\R\backslash \{0\})$ and satisfies
\bel{KK1}
-K(x)~=~K(-x),\qquad \left|K^{(i)}(x)\right|~\leq~{C_K\over |x|^{i+1}}\qquad \text{for} \, i=0,1.
\eeq
%
\end{itemize}
Under assumption {\bf (A1)}, the singular integral of convolution type ${\bf G}:{\bf L}^{2}(\R)\to {\bf L}^2(\R)$ is a bounded linear operator such that 
\bel{pend}
\langle {\bf G}[g],g\rangle_{{\bf L}^2(\R)}~=~0\qquad\forall g\in {\bf L}^2(\R).
\eeq
In addition, it holds 
\bel{Lq-b}
\|{\bf G}[g]\|_{{\bf L}^6(\R)}~\leq~C_{{\bf G}}\cdot\|g\|_{{\bf L}^6(\R)}\qquad\forall g\in {\bf L}^6(\R), \,\, \text{for some} \,\, C_{{\bf G}}>0.
\eeq

\begin{definition} 
We say that the Cauchy problem (\ref{BL})-(\ref{G}) with initial condition (\ref{ini-H}) exhibits wave breaking (shock formation) at time $T>0$ if it admits a  unique classical solution $u$ which is defined in $[0,T)\times\R$ and satisfies
\[
\liminf_{t\to T^-}\Big(\inf_{x\in \R}u_x(t,x)\Big)~=~-\infty.
\]
\end{definition}
Toward to the wave breaking  phenomenon of  (\ref{BL})-(\ref{G}), we first derive some basic ODEs related to the ${\bf L}^2$-norm of $u$, $u_x$, and $u_{xx}$.  Assume that  $u\in \mathcal{C}\big([0,T),H^2(\R)\big)\cap \mathcal{C}^1\big([0,T),H^1(\R)\big)$ is the solution of (\ref{BL}), (\ref{ini-H}). From (\ref{BL}) and (\ref{pend}), we have
\[
\frac12 {d\over dt}\int_{\R}u^2dx~=~\int_{\R}{\bf G}[u]udx~=~0,
\]
and this yields 
\bel{L-2}
\|u(t,\cdot)\|_{{\bf L}^2(\R)}~=~\|\bar{u}\|_{{\bf L}^2(\R)},\qquad t\in [0,T].
\eeq
 Differentiating (\ref{BL}) with respect to $x$, we  get
\bel{ux}
(u_x)_t+f'(u) u_{xx}+f''(u)u^2_x~=~ {\bf G}[u_x].
\eeq
For $q\geq 1$, multiplying (\ref{ux}) by $u_x|u_x|^{q-1}$, we get
\[
{d\over dt}\big|u_x\big|^{q+1}+ \big(f'(u)|u_x|^{q+1}\big)_x+q f''(u)u_x|u_x|^{q+1}~=~(q+1) {\bf G}[u_x] u_x|u_x|^{q-1},
\]
which implies
\bel{u-x-Lq}
{d\over dt}\int_{\R} \big|u_x\big|^{q+1}dx~=~-q\int_{\R} [f'(u)]_x |u_x|^{q+1}dx+(q+1)\int_{\R}{\bf G}[u_x] u_x|u_x|^{q-1}dx.
\eeq
Similarly, differentiating (\ref{BL}) twice with respect to $x$ and multiplying  by $2\, u_{xx}$, we obtain 
\[
\left(u^2_{xx}\right)_t+5f''(u)u_x u^2_{xx}+[f'(u)u^2_{xx}]_x+2f'''(u)u_x^3u_{xx}~=~2{{\bf G}[u_{xx}]}\, u_{xx},
\]
and this yields 
\bel{u-xx-L2}
{d\over dt}\|u_{xx}(t,\cdot)\|^2_{{\bf L}^2(\R)}~=~-5\int_{\R}[f'(u)]_x u_{xx}^2dx-2\int_{\R}f'''(u)u_x^3u_{xx} dx.
\eeq
The term $[f'(u)]_x$ and the last term in \eqref{u-xx-L2} complicate the analysis of wave breaking for the system \eqref{BL}--\eqref{G}. 

\subsection{Fluxes with growth condition} To obtain a rigorous argument, we begin by assuming that $f'''$ satisfies an appropriate growth condition, which includes, in particular, the case of the Burgers-Hilbert equation. 
\begin{itemize}
\item[({\bf A2})] The flux $f \in \mathcal{C}^3(\mathbb{R})$ is strictly convex and satisfies the following growth condition: there exist constants $\Gamma \ge 0$ and $p\in [0,2/3)$ such that
\begin{equation}\label{Ge-f}
|f'''(u)| \le \Gamma \big(1 + |u|^p\big), \quad \forall\, u \in \mathbb{R}.
\end{equation}
\end{itemize}
For every $p \in [0,2/3)$ and $v \in H^2(\mathbb{R})$, we denote by 
\bel{cons}
\theta_p~=~{2-3p\over 16}~\in~(0,1/8]\qquad\mathrm{and}\qquad \eta(v)~=~\eta^{1/8}_1(v)\cdot \eta^{1/2}_2(v),
\eeq
\bel{cons1}
\eta_1(v)~=~\|v'\|^2_{{\bf L}^2(\R)}\|v''\|^2_{{\bf L}^2(\R)}+2\Gamma\|v'\|^6_{{\bf L}^6(\R)}\, \left(\|v'\|^2_{{\bf L}^2(\R)}+2^{p} \|v\|^p_{{\bf L}^2(\R)} \|v'\|^{2+p}_{{\bf L}^2(\R)}\right),
\eeq
\bel{cons2}
\eta_2(v)~=~2^{1/2}f''(0)+\Gamma \left(4\|v\|^{1/2}_{{\bf L}^2(\R)}\|v'\|^{1/2}_{{\bf L}^2(\R)}+{p+2\over p+1} 2^{1+p/2}\|v\|^{(p+1)/2}_{{\bf L}^2(\R)}\|v'\|^{(p+1)/2}_{{\bf L}^2(\R)}\right). 
\eeq

Our main wave-breaking result for (\ref{BL}) is stated as follows:
\begin{theorem}\label{main4}
Assume that conditions {\bf (A1)} and {\bf (A2)} hold. Then, for every $\theta\in (0,\theta_p)$ and $\bar{u}\in H^2(\R)$ with 
\bel{B-Con}
\Big|\inf_{x\in \R}\big[f'\big(\bar{u}(x)\big)\big]'\Big|~>~\max\left\{{2(1+6C_{{\bf G}}+\Gamma)\over (2-3p)-16\theta}, {\eta(\bar{u})\over \theta^{1/2}}\right\},
\eeq
the Cauchy problem(\ref{BL})-(\ref{G}) with initial condition (\ref{ini-H}) exhibits wave breaking at some time $T^*>0$ such that 
\bel{b-T*}
{1\over (1+\theta)}\cdot {1\over \Big|\ds\inf_{x\in\R}[f'(\bar{u})]'(x)\Big|}~\leq~T^*~\leq~{1\over (1-\theta)}\cdot {1\over \Big|\ds \inf_{x\in\R}[f'(\bar{u})]'(x)\Big|}\,.
\eeq
\end{theorem}
{\bf Proof.} {\bf 1.} Since the linear operator ${\bf G}:{\bf L}^{2}(\R)\to {\bf L}^2(\R)$ is bounded and satisfies {\bf (A1)}, it is known from \cite{K}  that the above Cauchy problem  admits  a unique classical solution $u\in \mathcal{C}\big([0,t),H^2(\R)\big)\cap \mathcal{C}^1\big([0,t),H^1(\R)\big)$ for some $t>0$. Hence, we denote by 
\bel{T*}
T^*~\doteq~\sup\{t\geq 0: (\ref{BL}),(\ref{G}),(\ref{ini-H})~\text{admits a unique solution}~ u\in\mathcal{C}\big([0,t],H^2(\R)\big)\}~>~0.
\eeq
For every $t\in [0,T^*)$, set 
\bel{qan1}
m(t)~\doteq~\Big|\ds\inf_{x\in\R}\big[f'\big(u(t,x)\big)\big]_x\Big|\qquad\mathrm{and}\qquad M~\doteq~\|\bar{u}\|_{{\bf L}^2(\R)}.
\eeq
Recalling (\ref{L-2}) and (\ref{Ge-f}), we get 
\[
\|u(t,\cdot)\|_{{\bf L}^2(\R)}~=~M,\quad m(t)~\leq~\Gamma_1\, \left(1+\|u(t,\cdot)\|^{p+1}_{{\bf L}^\infty(\R)}\right)\Big|\inf_{x\in \R} u_x(t,x)\Big|,
\]
for some constant $\Gamma_1>0$. Using the same argument as in the proof of \cite[Lemma 3]{FAO}, one can show that 
\bel{u-l-b}
\begin{split}
\|u(t,\cdot)\|_{{\bf L}^\infty(\R)}&~\leq~\left[3\cdot \left(-\inf_{x\in \R} u_x(t,x)\right)\cdot \|u(t,\cdot)\|^{2}_{{\bf L}^2(\R)} \right]^{1/3}~=~\left[3M^2 \Big|\ds\inf_{x\in \R} u_x(t,x)\Big| \right]^{1/3},
\end{split}
\eeq
and this yields 
\bel{bb-m}
 m(t)~\leq~\Gamma_1\, \left(1+3^{(p+1)/3}M^{2(p+1)/3}\Big|\inf_{x\in \R} u_x(t,x)\Big|^{p+1\over 3}\right)\cdot \Big|\inf_{x\in \R} u_x(t,x)\Big|.
\eeq
On the other hand, from (\ref{u-x-Lq}), (\ref{pend}), and  (\ref{Lq-b}), it holds
\bel{ODE111}
{d\over dt}\|u_x(t,\cdot)\|^2_{{\bf L}^2(\R)}~\leq~m(t)\cdot\|u_x(t,\cdot)\|^2_{{\bf L}^2(\R)},\quad {d\over dt}\|u_x(t,\cdot)\|^6_{{\bf L}^6(\R)}~\leq~\big(5m(t)+6C_{{\bf G}}\big)\cdot \|u_x(t,\cdot)\|^6_{{\bf L}^6(\R)}.
\eeq
Calling $z_1,z_6$  the solution of 
\bel{ODE-1}
\begin{cases}
\dot{z}_1(t)~=~m(t)z_1(t),\qquad z_1(0)~=~\|\bar{u}'\|^2_{{\bf L}^2(\R)},\\[2mm]
 \dot{z}_6(t)~=~\big(5m(t)+6C_{{\bf G}}\big) z_6(t),\qquad z_6(0)~=~\|\bar{u}'\|^6_{{\bf L}^6(\R)},
\end{cases}
\eeq
we have 
\bel{comp-1}
\|u_x(t,\cdot)\|^2_{{\bf L}^2(\R)}~\leq~z_1(t),\qquad \|u_x(t,\cdot)\|^6_{{\bf L}^6(\R)}~\leq~z_6(t),\qquad t\in [0,T^*].
\eeq
Next, by (\ref{qan1}), we estimate  
\[
\|u(t,\cdot)\|_{{\bf L}^{\infty}(\R)}~\leq~2^{1/2}\|u(t,\cdot)\|^{1/2}_{{\bf L}^{2}(\R)}\|u_x(t,\cdot)\|^{1/2}_{{\bf L}^{2}(\R)}~\leq~2^{1/2}M^{1/2}z_1^{1/4}(t)
\]
\bel{ux-E}
\|u_x(t,\cdot)\|_{{\bf L}^{\infty}(\R)}~\leq~2^{1/2}\|u_x(t,\cdot)\|^{1/2}_{{\bf L}^{2}(\R)}\|u_{xx}(t,\cdot)\|^{1/2}_{{\bf L}^{2}(\R)}~\leq~2^{1/2}z_1^{1/4}(t)\|u_{xx}(t,\cdot)\|^{1/2}_{{\bf L}^{2}(\R)},
\eeq
and (\ref{u-xx-L2})-(\ref{Ge-f}) yield
\bel{ODE-2}
\begin{split}
{d\over dt} &\|u_{xx}(t,\cdot)\|^2_{{\bf L}^{2}(\R)}~\leq~5m(t)\, \|u_{xx}(t,\cdot)\|^2_{{\bf L}^{2}(\R)}+2\Gamma\int_{\R}(1+|u|^p) |u_x|^3|u_{xx}|dx\\[1mm]
 &~\leq~5m(t)\, \|u_{xx}(t,\cdot)\|^2_{{\bf L}^{2}(\R)}+2\Gamma\, \big(1+2^{p/2}M^{p/2}z_1^{p/4}(t)\big)\, z_6^{1/2}(t)\, \|u_{xx}(t,\cdot)\|_{{\bf L}^{2}(\R)}\\[2mm]
 &~\leq~(5m(t)+\Gamma)\, \|u_{xx}(t,\cdot)\|^2_{{\bf L}^{2}(\R)}+\Gamma\, \big(2+2^{p+1}M^{p}z_1^{p/2}(t)\big)z_6(t)
\end{split}
\eeq
Calling $\tilde{z}_2$  the solution of the ODE
\[
\dot{\tilde{z}}_2(t)~=~(5m(t)+\Gamma)\tilde{z}_2(t)+\Gamma\, \big(2+2^{p+1}M^{p}z_1^{p/2}(t)\big)z_6(t),\qquad \tilde{z}_2(0)~=~\|\bar{u}''\|^2_{{\bf L}^2(\R)},
\]
we have 
\[
\|u_{xx}(t,\cdot)\|^2_{{\bf L}^{2}(\R)}~\leq~\tilde{z}_2(t),\qquad t\in [0,T^*].
\]
Thus, we set 
\bel{Z1Z2}
Z_1~=~z_1\tilde{z}_2,\qquad Z_2~=~2z_1z_6+2^{p+1}M^{p}z_1^{1+p/2}z_6.
\eeq
By (\ref{ODE-1}) and (\ref{ODE-2}), we obtain  
\bel{Sys-1}
\begin{cases}
\dot{Z}_1(t)~=~(6m(t)+\Gamma)Z_1(t)+\Gamma Z_2(t),\\[1mm]
\dot{Z}_2(t)~\leq~\ds\big[(6+p/2) m(t)+6C_{{\bf G}}\big] Z_2(t),
\end{cases}
\qquad t\in (0,T^*).
\eeq
In particular, as long as $m$ remains bounded, both $Z_1$ and $Z_2$ stay bounded. Consequently, $\|u(t,\cdot)\|_{H^2(\mathbb{R})}$ also remains bounded.
Hence,  (\ref{bb-m}) implies that  $T^*$ defined in (\ref{T*}) corresponds to the wave breaking time.
\medskip

\n {\bf 2.} Next, let $v(t,x)$ be  the speed of the characteristic at $(t,x)$ defined by 
$$v(t,x)~\doteq~ f'\big(u(t,x)\big)\qquad\forall (t,x)\in [0,T^*]\times\R.$$
 Using (\ref{BL}), we  compute 
\[
v_t~=~f''(u)u_t~=~f''(u)\big(-f'(u)u_x+{\bf G}[u]\big)~=~-f'(u)v_x+f''(u){\bf G}[u].
\]
Differentiating the above equation with respect to $x$, we get
\[
(v_x)_t+f'(u) v_{xx}+v^2_x~=~f'''(u)u_x{\bf G}[u]+f''(u){\bf G}[u_x].
\]
Hence, for a given $\beta\in \R$, let $t\mapsto x(t;\beta)$ be the  characteristic starting from $\beta$ at time $t=0$ which solves the ODE
\bel{Cha-s}
\dot{x}(t;\beta)~=~f'\big(u(t,\beta)\big),\qquad x(0;\beta)~=~\beta.
\eeq
The map $t\mapsto w(t;\beta)$ which is defined as
\bel{w-f}
w(t;\beta)~=~-v_x\big(t,x(t;\beta)\big),
\eeq solves the ODE
\bel{w-b}
\begin{cases}
\dot{w}(t;\beta)~=~w^2(t;\beta)-\big[f'''(u)u_x{\bf G}[u]+f''(u){\bf G}[u_x]\big](t,x(t;\beta)),\\[2mm]
 w(0;\beta)~=~-f''\big(\bar{u}(\beta)\big)\bar{u}'(\beta),
\end{cases}
\eeq
and 
\bel{m-w}
\sup_{\beta\in \R}w(t;\beta)~=~m(t),\qquad \sup_{\beta\in \R}w^2(t;\beta)~\geq~m^2(t)\qquad\forall t\in [0,T^*).
\eeq
By  (\ref{L-2}) and (\ref{qan1}), we have 
\bel{G-b-u}
\|{\bf G}[u(t,\cdot)]\|_{{\bf L}^{\infty}(\R)}~\leq~2^{1/2}\|{\bf G}[u(t,\cdot)]\|^{1/2}_{{\bf L}^{2}(\R)}\|{\bf G}[u_x(t,\cdot)]\|^{1/2}_{{\bf L}^{2}(\R)}~\leq~2^{1/2}M^{1/2}z_1^{1/4}(t),
\eeq
\bel{G-b-ux}
\|{\bf G}[u_x(t,\cdot)]\|_{{\bf L}^{\infty}(\R)}~\leq~2^{1/2}\|{\bf G}[u_x(t,\cdot)]\|^{1/2}_{{\bf L}^{2}(\R)}\|{\bf G}[u_{xx}(t,\cdot)]\|^{1/2}_{{\bf L}^{2}(\R)}~\leq~2^{1/2}z_1^{1/4}(t)\tilde{z}_2^{1/4}(t).
\eeq
Using (\ref{ux-E}), (\ref{Z1Z2}), and (\ref{Ge-f}), we estimate
\[
\begin{split}
\big|f'''(u)u_x{\bf G}[u]&+f''(u){\bf G}[u_x]\big|\\[2mm]
& \leq~ \Gamma (1+|u|^p) |u_x| |{\bf G}[u]| + \bigg[f''(0) + \Gamma \, \Big(|u| + \frac{|u|^{p+1}}{p+1}\Big) \bigg] \, |{\bf G}[u_x]| \\
& \leq~2^{1/2} \tilde{z}^{1/4}_2 z_1^{1/4}\, \left[f''(0)+\Gamma \left(2^{3/2}M^{1/2} z_1^{1/4}+{p+2 \over p+1} 2^{(p+1)/2} M^{(p+1)/2} z_1^{(p+1)/4}\right)\right] \\
& =~2^{1/2}Z_1^{1/4}\, \left[f''(0)+\Gamma \left(2^{3/2}M^{1/2}z_1^{1/4}+{p+2 \over p+1}2^{(p+1)/2}M^{(p+1)/2} z_1^{(p+1)/4}\right)\right]. 
\end{split}
\]
For simplicity,  we introduce the nondecreasing  map $t\mapsto Z(t)$ which is defined by  
\bel{Z}
Z(t)~=~\left(Z_1+\Gamma Z_2\right)^{1/4}\, \left(\alpha_1+\alpha_2 z_1^{1/4}+\alpha_3z_1^{(p+1)/4}\right),
\eeq
with the constants  
\bel{al}
\alpha_1~=~2^{1/2}f''(0),\qquad \alpha_2~=~4\Gamma M^{1/2},\qquad \alpha_3~=~{p+2 \over p+1} \, 2^{1+p/2} \,\Gamma \, M^{(p+1)/2}.
\eeq
Then, by (\ref{cons})-(\ref{cons2})  and (\ref{w-b}) we obtain
\bel{w-beta}
Z^{1/2}(0)~=~\eta(\bar{u}),\qquad \big|\dot{w}(t;\beta)-w^2(t;\beta)\big|~\leq~Z(t),\qquad t\in (0,T^*).
\eeq
In the following steps, we shall use ODEs (\ref{Sys-1}) and (\ref{w-beta}) to show that $m$ goes to $+\infty$ in finite time and this yields the wave breaking for (\ref{BL}), (\ref{ini-H}).
\medskip

{\bf 3.} By assumption (\ref{B-Con}) and (\ref{w-beta}), we have
\[
\begin{split}
m(0)~=~\Big|\ds\inf_{x\in\R}\big[f'\big(\bar{u}(x)\big)\big]'\Big|&~>~\max\left\{{2(1+6C_{{\bf G}}+\Gamma)\over (2-3p)-16\theta},  {Z^{1/2}(0)\over \theta^{1/2}}\right\},
\end{split}
\]
and this implies that 
\bel{t111}
t_1~=~\sup\left\{t\in [0,T^*):m(s)>\max\left\{{2(1+6C_{{\bf G}}+\Gamma)\over (2-3p)-16\theta},  {Z^{1/2}(s)\over \theta^{1/2}}\right\}~~\forall s\in [0,t]\right\}~>~0.
\eeq
For every $(t,\beta)\in (0,t_1)\times\R$, it holds
\[
\begin{split}
{d\over dt}\left(w(t;\beta)- {Z^{1/2}\over \theta^{1/2}}\right)&~\geq~w^2(t;\beta)-\theta m^2-{1-\theta \over \theta^{1/2}}\,  mZ^{1/2}(t)~\geq~w^2(t;\beta)-m^2(t),
\end{split}
\]
In particular, for every $0<s<t_1-t$, we have 
\[
\begin{split}
m(t+s)-{Z^{1/2}(t+s)\over \theta^{1/2}}&~\geq~w(t+s;\beta)-{Z^{1/2}(t+s)\over \theta^{1/2}}\\[2mm]
&~\geq~w(t;\beta)-{Z^{1/2}(t)\over \theta^{1/2}}+\int_{t}^{t+s}\left(w^2(\tau;\beta)-m^2(\tau)\right)d\tau\\
&~=~w(t;\beta)-{Z^{1/2}(t)\over \theta^{1/2}}+ s \big[w^2(t;\beta)-m^2(t)\big]+ o(s),
\end{split}
\]
and (\ref{m-w}) yields 
\[
m(t+s)-{Z^{1/2}(t+s)\over \theta^{1/2}}~\geq~m(t)-{Z^{1/2}(t)\over \theta^{1/2}}+ o(s).
\]
By the Lipschitz continuity property of the map $t\mapsto m(t)$, we get
\[
{d\over dt} \left(m(t)-{Z^{1/2}(t)\over \theta^{1/2}}\right)~\geq~0\quad a.e.~t\in [0,t_1),
\]
which implies   
\[
m(t_1)- {Z^{1/2}(t_1)\over \theta^{1/2}}~\geq~m(0)- {Z^{1/2}(0)\over \theta^{1/2}}~>~0,\qquad m(t_1)~\geq~m(0)~>~{2(1+6C_{{\bf G}}+\Gamma)\over (2-3p)-16\theta}.
\]
From the definition of $t_1$ in (\ref{t111}), it follows that $t_1 = T^*$ and
\begin{equation}\label{Key-s2}
m(t)~>~ \frac{Z^{1/2}(t)}{\theta^{1/2}} \quad \forall\, t \in [0,T^*).
\end{equation}
Moreover, (\ref{w-beta}) yields
\begin{equation}\label{kes-1}
w^2(t;\beta) - \theta m^2(t)~\le~ \dot{w}(t;\beta) \le w^2(t;\beta) + \theta m^2(t)
\end{equation}
for all $\beta\in \R$ and a.e~ $t\in (0,T^*)$
\medskip

\n {\bf 4.} With the same argument in  Step 3, we deduce from the first inequality in (\ref{kes-1}) that
\[
\dot{m}(t) \ge (1-\theta)\, m^2(t) \quad \text{for a.e.\ } t \in [0,T^*).
\]
By a standard comparison argument, $m$ is bounded below by $m_1$ which is the solution of  
\bel{m1}
\dot{m}_1(t)~=~(1-\theta)\, m_1^2(t),\qquad m_1(0)=m(0).
\eeq
Solving the above ODE, we get 
\bel{low-b}
m(t)~\geq~m_1(t)~=~{m(0)\over 1-(1-\theta)t\, m(0)}
,\qquad t\in [0,T^*),
\eeq
and this  yields the upper bound of $T^*$ in (\ref{b-T*}).

Finally, using the second inequality in \eqref{kes-1}, for every $0 <  t-s<t < T^*$, we obtain
\[
\begin{aligned}
w(t;\beta)
&~\le~ m(t-s) + \int_{t-s}^{t} \big( w^2(\tau;\beta) + \theta m^2(\tau) \big)\, d\tau \\[2mm]
&~\le~ m(t-s) + s \big( w^2(t;\beta) + \theta  m^2(t) \big) + o(s).
\end{aligned}
\]
Together with \eqref{m-w}, this implies
\[
\dot{m}(t)~ \le~ (1+\theta)\, m(t) \quad \text{for a.e.\ } t \in (0,T^*).
\]
Hence, we get 
\begin{equation}\label{low-b}
m(t)~\le~ \frac{m(0)}{1-(1+\theta)t\, m(0)} \qquad \forall t \in [0,T^*),
\end{equation}
which gives the lower bound for $T^*$ in \eqref{b-T*}.
\endproof
\v
%

\begin{remark}\label{Lset}
For every given  $T>0$, we consider
\[
\mathcal{B}_T~=~\left\{v\in H^2(\R):\Big|\inf_{x\in\R}\big[f'\big(v(x)\big)\big]'\Big|> \max\left\{{(1+6C_{{\bf G}}+\Gamma)\over 4\theta_p}, {\sqrt{2}\cdot \eta(v)\over \sqrt{\theta_p}},{2\over (2-\theta_p)T}\right\}
\right\}.
\]
with $\eta$ defined in (\ref{cons})-(\ref{cons2}) and the constant $\theta_p=(2-3p)/16\in (0,1/8]$. Applying Theorem \ref{main4} for $\theta=\theta_p/2$, we obtain that the Cauchy problem (\ref{BL})-(\ref{G})
with initial data $\bar{u} \in \mathcal{B}_T$ exhibits wave breaking at time 
\[
T^*~\leq~{1\over (1-\theta_p/2)}\cdot {1\over \Big|\ds\inf_{x\in\R}[f'(\bar{u})]'(x)\Big|}~<~T.
\]
Hence, by (\ref{Ge-f}) and (\ref{cons})-(\ref{cons2}), the maps $v\mapsto [f'(v)]'$ and $v\mapsto \eta(v)$  are continuous in $H^2(\R)$, which implies that  the set $\mathcal{B}_T$ is an open subset of $H^2(\R)$. On the other hand, the open set $\mathcal{B}_T$ is quite large, since for every $v\in H^2(\R)\backslash\{0\}$ with $ \ds\inf_{x\in\R}\big[f'\big(v(x)\big)\big]'<0$, there exist constants $\gamma_0,\lambda_0>0$ such that the set 
\[
F_{v}~=~\{\lambda v \lor \gamma\, v(\lambda x): \gamma>\gamma_0 ,\lambda\geq \lambda_0 \}
\]
is contained in $\mathcal{B}_T$.
\end{remark}
To conclude this subsection, we shall establish the wave breaking of (\ref{BL})-(\ref{G}) for fluxes $f$ of quadratic type as in (\ref{f-q}), for which the analysis is more clear and the resulting estimates are sharper.
\medskip

\n {\bf Proof of Theorem \ref{BHG1}.} Assume that  $f$ is in a quadratic form as in  (\ref{f-q}). Then we have 
\[
f'''(u)~=~0,\qquad f''(u)~=~2a\qquad\forall u\in \R,
\]
which implies that 
\bel{m11}
m(t)~=~-\inf_{x\in\R}\big[f'\big(u(t,x)\big)\big]_x~=~-2a\inf_{x\in \R}u_x(t,x),\qquad\forall t\in [0,T^*).
\eeq
Moreover, (\ref{ODE111}), (\ref{ODE-2}), and (\ref{w-b}) becomes 
\[
{d\over dt} \|u_x(t,\cdot)\|^2_{{\bf L}^2(\R)}~\leq~m(t)\, \|u_x(t,\cdot)\|^2_{{\bf L}^2(\R)},\qquad {d\over dt} \|u_{xx}(t,\cdot)\|^2_{{\bf L}^2(\R)}~\leq~5m(t)\, \|u_{xx}(t,\cdot)\|^2_{{\bf L}^2(\R)},
\]
\[
\dot{w}(t;\beta)~=~w^2(t;\beta)-2a{\bf G}[u_x]\big(t,x(t;\beta)\big),\qquad  w(0;\beta)~=~ -2a \bar{u}'(\beta),\qquad \beta\in \R.
\]
Hence, consider the map $t\mapsto Z(t)$ defined by 
\[
Z(t)~=~2^{3/2}a\cdot \|u_x(t,\cdot)\|^{1/2}_{{\bf L}^2(\R)}\cdot \|u_{xx}(t,\cdot)\|^{1/2}_{{\bf L}^2(\R)}.
\] 
Recalling (\ref{G-b-ux}), we get 
\bel{SYS1}
\begin{cases}
\ds\dot{Z}(t)~\leq~{3m(t)\over 2}Z(t),\qquad Z(0)~=~2^{3/2}a \|\bar{u}'\|^{1/2}_{{\bf L}^2(\R)}\|\bar{u}''\|^{1/2}_{{\bf L}^2(\R)} \\[3mm]
\big|\dot{w}(t;\beta)-w^2(t;\beta)\big|~\leq~Z(t),\qquad w(0;\beta)~=~ -2a\bar{u}'(\beta).
\end{cases}
\eeq
Using the same argument as in the proof of Theorem \ref{main4}, from the condition on the initial data $\bar{u}$ in (\ref{B-Con-0}), we deduce that
\[
m(t)~>~\frac{Z^{1/2}(t)}{\theta^{1/2}}, \qquad \forall\, t \in [0,T^*),
\]
which implies (\ref{kes-1}). Hence, following steps 3 and 4 in the proof of Theorem \ref{main4}, we obtain
\[
\frac{m(0)}{1-(1-\theta)t\, m(0)}~\leq~m(t)~\leq~\frac{m(0)}{1-(1+\theta)t\, m(0)}, \qquad t \in [0,T^*),
\]
and (\ref{m11}) yields (\ref{b-T*-0}).
\endproof
\subsection{General strictly convex fluxes} 
In this subsection, we shall study the wave breaking of (\ref{BL})-(\ref{G}) for general strictly convex flux $f\in \mathcal{C}^3(\R)$. For $v \in H^2(\mathbb{R})$, we denote by 
\[
\lambda_1(v)~=~2^{1/2}\|v'\|^{1/4}_{{\bf L}^2(\R)} \left(\alpha_2(v)+\alpha_3(v) \|v\|^{1/2}_{{\bf L}^2(\R)} \|v'\|^{1/2}_{{\bf L}^2(\R)}\right)^{1/2}\left(\|v''\|^2_{{\bf L}^2(\R)}+\alpha_3(v)\|v'\|^6_{{\bf L}^6(\R)}\right)^{1/8},
\]
\[
\lambda_2(v) ~=~ {3^6 2C^3_K\|v\|^2_{{\bf L}^2(\R)}\over  \alpha_2(v) \|v\|^3_{{\bf L}^{\infty}(\R)}},\qquad \alpha_i(v)~=~\max_{|\omega|\leq 2\|v\|_{{\bf L}^\infty(\R)}}|f^{(i)}(\omega)|,\qquad i=2,3.
\]Our result is stated as follows:
\begin{theorem}\label{general-f}
Under the hypothesis {\bf (A1)}, assume that $f\in \mathcal{C}^3(\R)$ is strictly convex. Then for every $\theta\in (0,1/8)$ and $\bar{u}\in H^2(\R)$ with 
\bel{B-Con-2}
\Big|\inf_{x\in \R}\big[f'\big(\bar{u}(x)\big)\big]'\Big|~>~\max\left\{\frac{\alpha_3+6C_{{\bf G}}+1}{1-8\theta}, {\lambda_1(\bar{u})\over \theta^{1/2}}, {\lambda_2(\bar{u}) \theta \over (1-\theta)^3}\right\},
\eeq
the Cauchy problem(\ref{BL})-(\ref{G}) with initial condition (\ref{ini-H}) exhibits wave breaking at some time $T^*>0$ such that 
\bel{b-T*-g}
{1\over (1+\theta)}\cdot {1\over \Big|\ds\inf_{x\in\R}[f'(\bar{u})]'(x)\Big|}~\leq~T^*~\leq~{1\over (1-\theta)}\cdot {1\over \Big|\ds\inf_{x\in\R}[f'(\bar{u})]'(x)\Big|}\,.
\eeq
\end{theorem}
{\bf Proof.} {\bf 1.} Recalling $T^*$ in \textup{(\ref{T*})}, we first assume that 
\begin{equation}\label{P-as}
\sup_{t \in [0, T^*)} \left( \sup_{x \in \mathbb{R}} |u(t,x)| \right) ~\le~ \alpha_1~=~2\|\bar{u}\|_{{\bf L}^{\infty}(\R)}
\end{equation}
For simplicity, we define 
\bel{f''-}
\alpha_2~\doteq~\max_{|\omega|\leq \alpha_1}|f''(\omega)|\qquad\mathrm{and}\qquad \alpha_3~\doteq~\max_{|\omega|\leq \alpha_1}|f'''(\omega)|.
\eeq
As in the proof of Theorem \ref{main4},  from (\ref{u-x-Lq}) and (\ref{u-xx-L2}), let $\big(z_1,\tilde{z}_2, z_6\big)$ be the solution of the Cauchy problem 
\bel{ODE-1-g}
\begin{cases}
\dot{z}_1(t)~=~m(t)z_1(t),\\[2mm]
\dot{\tilde{z}}_2(t)~=~(5m(t)+\alpha_3) \tilde{z}_2(t)+\alpha_3\cdot z_6(t),\\[2mm]
 \dot{z}_6(t)~=~\big(5m(t)+6C_{{\bf G}}\big) z_6(t),
\end{cases}
\qquad
\begin{cases}
 z_1(0)~=~\|\bar{u}'\|^2_{{\bf L}^2(\R)},\\[2mm]
\tilde{z}_2(0)~=~\|\bar{u}''\|^2_{{\bf L}^2(\R)},\\[2mm]
 z_6(0)~=~\|\bar{u}'\|^6_{{\bf L}^6(\R)}.
\end{cases}
\eeq
By a comparison principle, we have 
\bel{bb-u}
\|u_x(t,\cdot)\|^2_{{\bf L}^2(\R)}~\leq~z_1(t),\quad \|u_{xx}(t,\cdot)\|^2_{{\bf L}^2(\R)}~\leq~\tilde{z}_2(t),\quad \|u_x(t,\cdot)\|^6_{{\bf L}^6(\R)}~\leq~z_6(t),\qquad t\in [0,T^*].
\eeq
For every $t\in [0,T^*)$ and $\beta\in \R$, we recall that $w(t;\beta)=-v_x(t;x(t;\beta))$ with $v=f'(u)$. By (\ref{w-b}) and (\ref{G-b-u})-(\ref{G-b-ux}), we derive 
\bel{w-1-b}
\begin{split}
\big|\dot{w}(t;\beta)-w^2(t;\beta)\big|&~\leq~\alpha_3\cdot \big|u_x {\bf G}[u]\big|+\alpha_2\cdot \big| {\bf G}[u_x]\big|~\leq~2\alpha_3 M^{1/2} z^{1/2}_1\tilde{z}_2^{1/4}+2^{1/2}\alpha_2 z_1^{1/4}\tilde{z}_2^{1/4}.
\end{split}
\eeq
From (\ref{ODE-1-g}), the map $t\mapsto (Z_1(t),Z_2(t))=(z_1(t)\tilde{z}_2(t), z_1(t)z_6(t))$ solves
\bel{Z12}
\begin{cases}
\dot{Z}_1~=~(6m(t)+\alpha_3)Z_1+\alpha_3 Z_2,\\[2mm]
 \dot{Z}_2~=~\big(6m(t)+6C_{{\bf G}}\big) Z_2,
\end{cases}
\qquad 
\begin{cases}
Z_1(0)~=~\|\bar{u}'\|^2_{{\bf L}^2(\R)}\|\bar{u}''\|^2_{{\bf L}^2(\R)},\\[2mm]
Z_2(0)~=~\|\bar{u}'\|^2_{{\bf L}^2(\R)}\|\bar{u}'\|^6_{{\bf L}^6(\R)}.
\end{cases}
\eeq
Introduce the nondecreasing  map $t\mapsto Z(t)$ which is defined by  
\bel{Zt-g}
Z~=~(Z_1+\alpha_3Z_2)^{1/4}\cdot \big(2^{1/2}\alpha_2+2\alpha_3 M^{1/2} z_1^{1/4}\big).
\eeq
Using (\ref{w-1-b}) and (\ref{Z12}), for $t\in (0,T^*)$, we obtain 
\bel{w-2-b}
Z^{1/2}(0)~=~\lambda_1(\bar{u}), \qquad \big|\dot{w}(t;\beta)-w^2(t;\beta)\big|~\leq~Z(t)\qquad\forall \beta\in \R,
\eeq
and 
\bel{DC-Z}
\begin{split}
\dot{Z}
&~=~{Z\big( [6m+\alpha_3]  Z_1+ [6m+6C_{\bf G}+1] \alpha_3 Z_2\big)\over 4 (Z_1+\alpha_3 Z_2)}+(Z_1+\alpha_3 Z_2)^{1/4}\cdot {m \alpha_3 M^{1/2} z_1^{1/4}\over 2}  \\
&~\leq~{Z\over 4}\cdot \left(7m+\alpha_3+6C_{\bf G}+1\right).
\end{split}
\eeq
{\bf 2.} By the assumption (\ref{B-Con-2}), we have 
\[
\begin{split}
m(0)~=~-\inf_{x\in\R}\big[f'\big(\bar{u}(x)\big)\big]'&~>~\max\left\{\frac{\alpha_3+6C_{{\bf G}}+1}{1-8\theta},  {Z^{1/2}(0)\over \theta^{1/2}}\right\},
\end{split}
\]
and this implies that 
\bel{t112}
t_1~=~\sup\left\{t\in [0,T^*):m(s)>\max\left\{\frac{\alpha_3+6C_{{\bf G}}+1}{1-8\theta},  {Z^{1/2}(s)\over \theta^{1/2}}\right\}~~\forall s\in [0,t]\right\}~>~0.
\eeq
Hence, for every $t\in (0,t_1)$, it holds
\[
\dot{Z}(t)~\leq~2 \left(1-\theta\right) m Z, \qquad  {d\over dt}\left(w(t;\beta)- {Z^{1/2}(t)\over \theta^{1/2}}\right)~\geq~w^2(t;\beta)-m^2(t).
\]
Using the same argument as in the proof of Theorem \ref{main4}, we obtain that 
\bel{ekey}
m(t)~>~{Z^{1/2}(t)\over \theta^{1/2}}\qquad\mathrm{and}\qquad \big|\dot{w}(t;\beta)-w^2(t;\beta)\big|~\leq~\theta m^2(t)\qquad\forall t\in [0,T^*).
\eeq
Hence, we get 
\[
\frac{m(0)}{1-(1-\theta)t\, m(0)}~\leq~m(t)~\leq~\frac{m(0)}{1-(1+\theta)t\, m(0)}, \qquad t \in [0,T^*),
\]
and this yields (\ref{b-T*-g}).
\medskip

{\bf 3.} To conclude the proof, we need to verify (\ref{P-as}). For every $\beta\in \R$ and $\eta>0$, we estimate 
\[
\begin{split}
\left|{d\over dt}u(t,x(t;\beta))\right|&~=~\left|{\bf G}[u(t,\cdot)](x(t;\beta))\right|~\leq~\int_{0}^{\eta} 2yK(y) {u(t,x(t;\beta)-y)-u(t,x(t;\beta)+y)\over 2y}~dy\\
&\qquad\qquad\qquad\quad +\int_{\eta}^{\infty}K(y) \big[u(t,x(t;\beta)-y)-u(t,x(t;\beta)+y)\big]dy\\
&~\leq~2C_K \big(\eta \|u_x (t,\cdot)\|_{{\bf L}^{\infty}(\R)}+\eta^{-1/2} \|u(t,\cdot)\|_{{\bf L}^2(\R)}\big)\\
&~\leq~2C_K \left(2\eta \|u_x (t,\cdot)\|^{1/2}_{{\bf L}^{2}(\R)} \|u_{xx} (t,\cdot)\|^{1/2}_{{\bf L}^{2}(\R)}+\eta^{-1/2}M\right).
\end{split}
\] 
Choosing $\eta=2^{-4/3} M^{2/3}\|u_x (t,\cdot)\|^{-1/3}_{{\bf L}^{2}(\R)} \|u_{xx} (t,\cdot)\|^{-1/3}_{{\bf L}^{2}(\R)}$, by (\ref{Zt-g}) we derive  
\[
\begin{split}
\left|{d\over dt}u(t,x(t;\beta))\right|&~\leq~2^{2/3}3C_KM^{2/3}\|u_x (t,\cdot)\|^{1/6}_{{\bf L}^{2}(\R)} \|u_{xx} (t,\cdot)\|^{1/6}_{{\bf L}^{2}(\R)}~\leq~2^{2/3}3C_KM^{2/3}Z^{1/12}_1(t)
\\
&~\leq~{2^{1/3}3C_KM^{2/3}\over \alpha_2^{1/3}}\cdot Z^{1/3}~\leq~ {2^{1/3}3C_KM^{2/3}\over \alpha_2^{1/3}}\cdot \theta^{1/3} m^{2/3}(t),
\end{split}
\]
which yields
\bel{u-b-l}
|u(t,x(t;\beta))|~\leq~\|\bar{u}\|_{{\bf L}^{\infty}(\R)}+{2^{1/3}3C_KM^{2/3}\over \alpha_2^{1/3}} \cdot \theta^{1/3} \int_{0}^{t}m^{2/3}(s)ds,\qquad t\in [0,T^*).
\eeq
For every $\ve>0$, we introduce the decreasing function $F_{\ve}:\R\to [0,\infty)$ such that 
\[
F_{\ve}(x)~=~\int_{x}^{\infty}{1\over y^{4/3}+\ve}dy,\qquad x\in \R.
\]
From (\ref{ekey}), for every $\beta\in \R$, one has
\[
\begin{split}
F_{\ve}(m(t+s))&~\leq~F_{\ve}(w(t+s;\beta))\\
&~=~F_{\ve}(w(t;\beta))-\int_{t}^{t+s}{\dot{w}(\tau;\beta)\over w^{4/3}(\tau,\beta)+\ve}d\tau\\
&~\leq~F_{\ve}(w(t;\beta))-\int_{t}^{t+s}{w^2(\tau;\beta)\over w^{4/3}(\tau,\beta)+\ve}- { \theta m^2(\tau)\over w^{4/3}(\tau,\beta)+\ve}d\tau\\
&~=~F_{\ve}(w(t;\beta))- s\cdot \left({w^2(t;\beta)\over w^{4/3}(t;\beta)+\ve}- { \theta m^2(t)\over w^{4/3}(t;\beta)+\ve}\right)+o(s),
\end{split}
\]
and this implies
\[
F_{\ve}(m(t+s))~\leq~F_{\ve}(m(t))-s\cdot  { (1-\theta) m^2(t)\over m^{4/3}(t)+\ve}+o(s).
\]
By the Lipschitz continuity of $m$, we get 
\[
{d\over dt} F_{\ve}(m(t))~\leq~ -{ (1-\theta) m^2(t)\over m^{4/3}(t)+\ve}\quad a.e.~ t\in (0,T^*),
\]
which gives 
\[
F_{\ve} (m(t))+(1-\theta)\cdot \int_{0}^{t}{m^2(s)\over m^{4/3}(s)+\ve}ds~\leq~F_{\ve}(m(0)).
\]
Since $\ds\lim_{\ve\to 0+}F_{\ve}(x)=3x^{-1/3}$ for all $x>0$, letting $\ve\to 0+$ yields
\[
3m^{-1/3}(t)+(1-\theta)\cdot\int_{0}^{t}m^{2/3}(s)~\leq~3m^{-1/3}(0)
\]
Finally, from (\ref{u-b-l}), we derive 
\[
|u(t,x(t;\beta))|~\leq~\|\bar{u}\|_{{\bf L}^{\infty}(\R)}+{2^{1/3} 3^2 C_KM^{2/3}\over \alpha_2^{1/3}(1-\theta)}\cdot \theta^{1/3} m^{-1/3}(0),
\]
and assumption (\ref{B-Con-2}) yields (\ref{P-as}). 
\endproof
\v

{\bf Acknowledgements.} This research was partially supported by NSF-DMS 2154201 (K.T. Nguyen).

\end{document}